\documentclass[
letterpaper,
  onecolumn,
  colorlinks,
]{preprint}

\usepackage{geometry}
\usepackage{graphicx}
\usepackage{amsmath}
\usepackage{amsthm}
\usepackage{amssymb}
\usepackage{algorithm}
\usepackage{algpseudocode}
\usepackage{bbm}
\usepackage{subcaption}
\usepackage{todonotes}

\usepackage{pgfplots}
\usepackage{tikz}
\usetikzlibrary{decorations.markings,patterns}
\usetikzlibrary{plotmarks}
\pgfplotsset{compat=1.18}

\mathcode`\@="8000
{\catcode`\@=\active \gdef@{\mkern1mu}} 

\algnewcommand\algorithmicofflinerequire{\textbf{Offline Input:}}
\algnewcommand\OfflineRequire{\item[\algorithmicofflinerequire]}
\algnewcommand\algorithmiconlinerequire{\textbf{Online Input:}}
\algnewcommand\OnlineRequire{\item[\algorithmiconlinerequire]}
\algnewcommand\algorithmicofflineensure{\textbf{Offline Output:}}
\algnewcommand\OfflineEnsure{\item[\algorithmicofflineensure]}
\algnewcommand\algorithmiconlineensure{\textbf{Online Output:}}
\algnewcommand\OnlineEnsure{\item[\algorithmiconlineensure]}

\algnewcommand\algorithmicstartoffline{\textbf{Offline Phase}}
\algnewcommand\StartOfflinePhase{\item[\algorithmicstartoffline]}
\algnewcommand\algorithmicstartonline{\textbf{Online Phase (for each \boldmath $\hp$\kern1pt)}}
\algnewcommand\StartOnlinePhase{\item[\algorithmicstartonline]}

\newcommand{\C}{\mathbb{C}}
\newcommand{\R}{\mathbb{R}}

\renewcommand{\L}{\mathbb{L}}
\newcommand{\hz}{\widehat{z}}
\renewcommand{\Re}{\operatorname{Re}}
\renewcommand{\Im}{\operatorname{Im}}
\renewcommand{\imath}{\mathrm{i}}
\newcommand{\dop}{\mathrm{d}}

\newcommand{\bA}{\mathbf{A}}
\newcommand{\bB}{\mathbf{B}}
\newcommand{\bC}{\mathbf{C}}

\newcommand{\bE}{\mathbf{E}}

\newcommand{\bG}{\mathbf{G}}
\newcommand{\bH}{\mathbf{H}}
\newcommand{\bI}{\mathbf{I}}
\newcommand{\bJ}{\mathbf{J}}

\newcommand{\bL}{\mathbf{L}}

\newcommand{\bN}{\mathbf{N}}

\newcommand{\bP}{\mathbf{P}}

\newcommand{\bR}{\mathbf{R}}
\newcommand{\bS}{\mathbf{S}}
\newcommand{\bT}{\mathbf{T}}

\newcommand{\bV}{\mathbf{V}}
\newcommand{\bW}{\mathbf{W}}
\newcommand{\bX}{\mathbf{X}}
\newcommand{\bY}{\mathbf{Y}}

\newcommand{\bb}{\mathbf{b}}
\newcommand{\bc}{\mathbf{c}}

\newcommand{\br}{\mathbf{r}}
\newcommand{\bs}{\mathbf{s}}

\newcommand{\bv}{\mathbf{v}}
\newcommand{\bw}{\mathbf{w}}
\newcommand{\bx}{\mathbf{x}}

\newcommand{\bzero}{\mathbf{0}}
\newcommand{\bell}{\boldsymbol{\ell}}
\newcommand{\bSigma}{\boldsymbol{\Sigma}}

\newcommand{\tbN}{\widetilde{\bN}}
\newcommand{\tbH}{\widetilde{\bH}}

\newcommand{\cD}{\mathcal{D}}
\newcommand{\cE}{\mathcal{E}}

\newcommand{\cO}{\mathcal{O}}
\newcommand{\cP}{\mathcal{P}}

\newcommand{\hp}{\widehat{p}}

\newcommand{\eop}{{\mathrm{e}}}

\newcommand{\pdom}{\Pi}       

\newtheorem{proposition}{Proposition}

\newtheorem{corollary}{Corollary}
\newtheorem{lemma}{Lemma}
\newtheorem{definition}{Definition}
\newtheorem{theorem}{Theorem}
\newtheorem{example}{Example}
\newtheorem{remark}{Remark}

\title{A parametric Keldysh decomposition}
\author{Linus Balicki, Mark Embree, Serkan Gugercin\thanks{Funded by US National Science Foundation grant DMS-2411141.\newline Department of Mathematics, 225 Stanger Street 0123, Virginia Tech, Blacksburg, VA 24061\newline (balicki@vt.edu, embree@vt.edu, gugercin@vt.edu)}}
\date{January 2026}

\begin{document}

\novelty{}

\abstract{Contour integral algorithms seek to compute a small number of eigenvalues located within a bounded region of the complex plane. These methods can be applied to both linear and nonlinear matrix eigenvalue problems.  In the latter case, the foundation of these methods comes from the Keldysh decomposition, which breaks the nonlinear matrix-valued function into two parts: a rational function whose poles match the desired eigenvalues, and a remainder term that is analytic within the target region.  Under contour integration this analytic part vanishes (via Cauchy's theorem), leaving only the component containing the desired eigenvalues.  We propose an extension of the Keldysh decomposition for matrix-valued functions that depend analytically on an additional parameter. We establish key properties of this parametric Keldysh decomposition, and introduce an algorithm for solving parametric nonlinear eigenvalue problems that is based upon it.
}

\maketitle

\section{Introduction}
Let $\bT:\C\rightarrow \C^{n\times n}$ denote a general matrix-valued function.
The \emph{nonlinear eigenvalue problem} (NLEVP) seeks to find an \emph{eigenvalue} $\lambda\in\C$ and nonzero \emph{eigenvector} $\bv\in\C^n$ such that
\begin{equation}
\label{eq:NLEVP}
    \bT(\lambda) \bv = \bzero.
\end{equation}
The linear eigenvalue problem $\bA \bv = \lambda \bv$ with $\bA \in \C^{n \times n}$ can be written in this form by setting $\bT(z) = z \bI - \bA$. Here, we seek all eigenvalues of $\bT$ located within some simply connected, bounded domain $\Omega\subset \C$, meaning that $\Omega$ is a domain (nonempty, open, connected) that is additionally bounded and simply connected. We assume that $\bT$ is analytic throughout $\Omega$ and its boundary $\partial \Omega$, and regular, i.e., $\det(\bT(\cdot)) \not \equiv 0$. The desired eigenvalues are the poles of $\bT(z)^{-1}$ that are located in $\Omega$.\ \  One can approach this problem through a variety of algorithmic strategies, including Newton's method, linearization, and contour integration, as surveyed by G\"uttel and Tisseur~\cite{guttel2017}. We focus on NLEVP solvers that are based on contour integration.  These methods originated with the Sakurai--Sugiura algorithm~\cite{SS03b} for the linear eigenvalue problem.  For simplicity, consider a Hermitian matrix $\bA\in\C^{n\times n}$ with $m$ eigenvalues $\lambda_1,\ldots, \lambda_m$ in $\Omega$ and the remaining eigenvalues $\lambda_{m+1},\ldots, \lambda_n$ exterior to the closure $\overline{\Omega}=\Omega \cup \partial\Omega$, corresponding to a set of orthonormal eigenvectors $\bv_1, \ldots, \bv_n$.\ \ Then the resolvent of $\bA$ can be decomposed as
\begin{equation} \label{eq:symm}
(z\bI-\bA)^{-1} 
= \bigg(\sum_{j=1}^m \frac{1}{z-\lambda_j} \bv_j^{}\bv_j^* \bigg)
 + \bigg(\sum_{j=m+1}^n \frac{1}{z-\lambda_j} \bv_j^{}\bv_j^* \bigg),
\end{equation}
where $\cdot^*$ denotes the conjugate-transpose.
The first sum in \eqref{eq:symm} contains all the poles of the resolvent (eigenvalues of $\bA$) in $\Omega$; the second sum is thus analytic in $\Omega$.  
By integrating around the boundary $\partial\Omega$ and invoking Cauchy's theorem, one can use functions $f$ that are analytic on $\Omega$ to learn about the desired eigenvalues:
\[ \frac{1}{2\pi@\imath} \int_{\partial \Omega} f(z) (z\bI-\bA)^{-1}\, \dop z 
      = \sum_{j=1}^m f(\lambda_j) \bv_j^{} \bv_j^*.
\]
Contour integral methods often use $f(z) = z^k$ to compute moments of the 
desired component of the resolvent.  Several variations proposed in~\cite{brennan2023} use $f(z) = 1/(z-\sigma)^k$, leading to connections with system identification and rational interpolation.

The class of contour integral eigensolvers was extended to NLEVPs by Asakura et al.~\cite{asakura2009} and Beyn~\cite{beyn2012}.  Indeed Beyn motivated his method by invoking a theorem developed by Keldysh%
\footnote{The result was announced in 1951~\cite{keldysh1951}, but only elaborated upon in 1971~\cite{keldysh1971}.  Gohberg~\cite{Goh89} provides historical context for this lapse of time, explaining that Keldysh was consumed during the intervening years with an important clandestine role in the Soviet space program.} 
in the context of linear operators \cite{keldysh1951,keldysh1971}, which can be seen as a generalization of the decomposition~\eqref{eq:symm}; see also the contemporaneous work of Gohberg and Sigal~\cite{GS71}.
When specialized to analytic matrix-valued functions, Keldysh's theorem expresses $\bT(z)^{-1}$ as
\begin{equation}
    \label{eq:Keldyshdecomposition}
    \bT(z)^{-1} = \bH(z) + \bN(z),
\end{equation}
where $\bN$ is analytic in $\Omega$ and $\bH$ is a rational function that encodes eigenvalue and eigenvector information in its poles and residues.  We will discuss details of Keldysh's theorem and how it can be leveraged for solving NLEVPs in the next section.

In many practical applications we are not only interested in solving a single NLEVP, but would like to investigate how eigenvalues and eigenvectors change as we vary a parameter that affects $\bT$.\ \ For example, a \emph{parameteric NLEVP} (pNLEVP) could depend on the time-delay in a delay differential equation, material properties in a mechanical design problem, or a Floquet parameter in a periodic system. To define a pNLEVP we consider $\bT : \C\times \C \rightarrow \C^{n \times n}$, and assume that we have a parameter set $\cP \subset \C$ and a spectral domain of interest $\Omega \subset \C$. Then we aim to find eigenvalues $\lambda(p) \in \Omega$ and nonzero eigenvectors $\bv(p) \in \C^n$ such that
\begin{equation}
    \label{eq:pNLEVP}
    \bT(\lambda(p),p) \bv(p) = \bzero
\end{equation}
for each parameter $p \in \cP$.\ \ One could approach this task by repeatedly applying any standard NLEVP solver to $\bT(@\cdot@,p):\C \rightarrow \C^{n \times n}$, fixing each desired value of~$p$. This approach will be computationally demanding, especially if we must perform it for many parameter values.  We want to avoid this repetitive procedure, which can also mask insight about how the eigenvalues and eigenvectors evolve with $p$.  We propose a more efficient approach that invests some initial work (offline phase) to compute a compact representation of the pNLEVP from which we can quickly extract approximations to $\lambda(p)$ and $\bv(p)$ for any given $p \in \cP$ (online phase).\ \  Our method is based on an extension of Keldysh's theorem for matrix-valued functions to the parameter-dependent setting. 
We derive a decomposition of the form
\begin{equation*}
    \bT(z,p)^{-1} = \bH(z,p) + \bN(z,p),
\end{equation*}
where for all $p \in \cP$ the poles of $\bH(\cdot,p)$ correspond to the eigenvalues $\lambda(p)$ in $\Omega$, and $\bN(\cdot,p)$ is analytic in $\Omega$.\ \ One of our key contributions is to specify properties of the functions $\bH$ and $\bN$ as functions of both $z$ and $p$ using techniques from the theory of several complex variables.\ \ This novel result provides a powerful theoretical tool for analyzing pNLEVPs and forms the basis for our proposed numerical scheme.

Many solution approaches for parametric linear and nonlinear eigenvalue problems have been proposed in the literature. 
For example, continuation-based methods seek to track individual or multiple eigenvalues (or invariant subspaces) as the parameter changes~\cite{betcke2011,beyn2011,beyn2010,bindel2008}. At a simplified level, these algorithms treat $\lambda(p)$ as a parameterized curve, using derivative information at past parameter values to approximate $\lambda(p)$ at a new value of $p$. The approximation is then corrected via an iterative algorithm. If $\lambda(p)$ is sufficiently smooth, one expects eigenvalue continuation to be effective. However, the global smoothness of eigenvalue curves is not always guaranteed,  with degeneracies sometimes arising at physically relevant parameter values.  Special care must be taken to handle non-differentiable behavior in continuation algorithms.  Another common solution approach for pNLEVPs leverages polynomial approximations for $\lambda$ \cite{andreev2012,mach2025,pradovera2024}. Again, smoothness of eigenvalues is crucial to guarantee the existence of good approximants. 

Motivated by the parametric Keldysh decomposition, 
our algorithmic approach differs from most of the previously mentioned methods, in that it \emph{approximates an implicit representation of the eigenvalues}. Our method approximates $\bH$ itself, which encodes the desired eigenvalues and eigenvectors, rather than computing an explicit parametrization of a curve of eigenvalues. For new parameter values the eigenvalues of interest can then quickly be extracted from $\bH$.\ \  This approach can deal effectively and automatically with cases where $\lambda$ is known to exhibit nonsmooth behavior. 

\section{Organization and background}
Our two main contributions are Theorem~\ref{theorem:parametricKeldysh}, which introduces the parametric Keldysh decomposition and establishes its key properties, and a numerical framework for solving pNLEVPs that is built upon it.  In this section we briefly recapitulate Keldysh's theorem (Section~\ref{sec:Keldysh}) and contour integral methods (Section~\ref{sec:NLEVP}) for non-parametric problems. Section~\ref{sec:parametricKeldysh} contains the parametric version of Keldysh's theorem. In Sections~\ref{sec:parametricNLEVP} and~\ref{sec:parametricMultipointLoewner} we set up and formulate a novel algorithm based on the parametric Keldysh decomposition for tackling pNLEVPs. The numerical examples in Section~\ref{sec:numericalExamples} demonstrate the effectiveness of our proposed methods. We summarize and highlight several possible directions for future research in Section~\ref{sec:conclusions}.

\subsection{The Keldysh decomposition}
\label{sec:Keldysh}
We begin by discussing Keldysh's theorem for the special case where $\bT$ is a matrix-valued function that has only simple eigenvalues.
See, for example,~\cite[thm.~2.8]{guttel2017}.
\begin{theorem}[Keldysh]
\label{theorem:keldysh}
Let $\bT : \C \rightarrow \C^{n \times n}$ be a matrix-valued function analytic in the domain $\Omega \subset \C$, and $\det(\bT(@\cdot@)) \not\equiv 0$. Suppose $\bT$ has $m$ simple eigenvalues in $\Omega$ denoted by $\lambda_1,\ldots,
\lambda_m$ with associated left and right eigenvectors forming the columns of $\bV = [\bv_1,\ldots,\bv_m] \in \C^{n \times m}$ and $\bW = [\bw_1,\ldots,\bw_m] \in \C^{n \times m}$. Then 
\begin{equation*}
    \bT(z)^{-1} = \bV (z \bI - \bJ)^{-1} \bW^{*} + \bN(z),
\end{equation*}
where $\bJ = \operatorname{diag}(\lambda_1,\ldots,\lambda_m)$ and $\bN$ is analytic in $\Omega$. The eigenvectors are normalized such that $\bw_j^*@ \bT@'(\lambda_j) \bv_j = 1$.
\end{theorem}
One of the key insights  from Keldysh's theorem is that $\bT(@\cdot@)^{-1}$ can be decomposed into the sum of two parts: 
a rational function
\begin{equation}
    \label{eq:HKeldysh}
    \bH(z) = \bV (z \bI - \bJ)^{-1} \bW^{*}
\end{equation}
that encodes information about the eigenvalues and eigenvectors in the domain of interest, 
and a function $\bN(z)$ that is analytic in $\Omega$ and contains no information about the desired eigenvalues.
 More precisely, $\bH$ has $m$ poles that correspond to the eigenvalues $\lambda_1,\ldots,\lambda_m$ and matrix-valued residues that reveal $\operatorname{Res}(\bH,\lambda_j) = \bv_j^{} \bw_j^*$ if $\lambda_j$ is a simple eigenvalue. In Section~\ref{sec:NLEVP} we review the multipoint Loewner framework introduced in \cite{brennan2023}, which reconstructs the individual matrices $\bJ$, $\bV$, and $\bW$ (i.e., the eigenvalues and eigenvectors of $\bT$ in $\Omega$) that constitute $\bH$, as in \eqref{eq:HKeldysh}.

Before moving to this computational framework, we note that Theorem~\ref{theorem:keldysh} can be generalized to handle defective eigenvalues (see Theorem~\ref{theorem:generalKeldysh} in Appendix~\ref{sec:generalKeldysh}, and additional details in~\cite{guttel2017}).  The diagonal matrix $\bJ$ is replaced by a Jordan matrix, and the columns of $\bV$ and $\bW$ are formed by generalized eigenvectors; the  eigenvector normalization condition becomes more intricate. This generalization is relevant for pNLEVPs, since defective eigenvalues appear at some parameter values in a wide range of examples. Our main results and algorithms are applicable even if complicated Jordan structures are present (as partially demonstrated via a $2 \times 2$ Jordan block in Example~\ref{example:simpleKeldysh} and the numerical treatment in Section~\ref{sec:numericalExamples}). The simplified presentation in Theorem~\ref{theorem:keldysh} is intended to convey the main idea clearly. 

\subsection{Solving NLEVPs via contour integration and rational interpolation}
\label{sec:NLEVP}
Here we revisit how the multipoint Loewner framework can be used 
to approximate eigenvalues and eigenvectors of $\bT$~\cite{brennan2023}.\ \  Given access only to $\bT(z)$, we can use the decomposition \cref{eq:Keldyshdecomposition} to evaluate $\bH(z)$.
The framework use two key ideas:
\begin{itemize}
    \item We can evaluate $\bH$ at $\theta \notin \overline{\Omega}$ by evaluating the integral
    \begin{equation*}
      \bH(\theta) \,=\,  \frac{1}{2\pi \imath}\int_{\partial\Omega} \frac{1}{\theta - z} \bT(z)^{-1} \,\dop z;
    \end{equation*}
    \item Using tools from systems identification, we can construct $\bJ,\bV$ and $\bW$ by rationally interpolating a (sufficiently rich) set of samples $\left\{ \bH(\theta_1),\ldots,\bH(\theta_r) \right\}$.
\end{itemize}
Let us focus on the first point by considering, for some $\theta \not\in\overline{\Omega}$,
\begin{equation*}
    \frac{1}{2\pi \imath}\int_{\partial\Omega} \frac{1}{\theta - z} \bT(z)^{-1}\,\dop z = \frac{1}{2\pi \imath}\int_{\partial\Omega} \frac{1}{\theta - z} \bH(z)\, \dop z + \frac{1}{2\pi \imath}\int_{\partial\Omega} \frac{1}{\theta - z} \bN(z)\,\dop z.
\end{equation*}
Note that both $\bN(z)$ and $1/(\theta - z)$ are analytic in $\Omega$, so Cauchy's  theorem yields
\begin{equation*}
    \frac{1}{2\pi \imath}\int_{\partial\Omega} \frac{1}{\theta - z} \bN(z)\, \dop z = 0.
\end{equation*}
Additionally, the Cauchy integral formula implies that
\begin{equation*}
    \frac{1}{2\pi \imath}\!\int_{\partial\Omega} \frac{1}{\theta\!-\!z} \bH(z)\,\dop z = \bV\!\! \left(\! \frac{1}{2\pi\imath}\! \int_{\partial\Omega} \frac{1}{\theta\!-\!z} (z\bI - \bJ)^{-1} \dop z \!\right)\!\! \bW^* = \bV(\theta \bI - \bJ)^{-1} \bW^* = \bH(\theta). 
\end{equation*}
Combining these observations gives
\begin{equation}
    \label{eq:Hcontourintegral}
    \bH(\theta) = \frac{1}{2\pi \imath}\int_{\partial\Omega} \frac{1}{\theta - z} \bT(z)^{-1}\, \dop z.
\end{equation}
One can use numerical quadrature to approximate the integral, giving approximate samples of $\bH$.\ \ For a quadrature rule with nodes $z_1,\ldots,z_N$ and weights $w_1,\ldots,w_N$, 
\begin{equation}
    \label{eq:Hquadrature}
    \bH(\theta) \;=\; \frac{1}{2\pi \imath}\int_{\partial\Omega} \frac{1}{\theta - z} \bT(z)^{-1}\,\dop z \;\approx\; \sum_{k=1}^N \frac{w_k}{\theta - z_k} \bT(z_k)^{-1}.
\end{equation}
The second key idea is to apply an important tool from systems theory, the Loewner framework of Mayo and Antoulas \cite{antoulas2017,mayo2007}.
Equation~(\ref{eq:Hquadrature}) suggests that we need to compute explicit inverses $\bT(z_k)^{-1}$.  Rather, the Loewner framework only requires \emph{tangential samples} of $\bH(z)$.
Suppose we have a set of left and right \emph{probing directions}
\begin{equation*}
    \bell_1,\ldots,\bell_r \in \C^n \qquad \text{and} \qquad \br_1,\ldots,\br_r \in \C^n,
\end{equation*}
as well as left and right \emph{sample points}
\begin{equation*}
    \theta_1,\ldots,\theta_r \in \C \setminus \overline{\Omega} \qquad \text{and} \qquad \sigma_1,\ldots,\sigma_r \in \C \setminus \overline{\Omega}.
\end{equation*}
Under mild conditions, the Loewner framework exactly recovers $\bH(z)$ from the data 
\begin{equation*}
    \bell_1^\top \bH(\theta_1),\ldots,\bell_r^\top \bH(\theta_r) \qquad \text{and} \qquad \bH(\sigma_1) \br_1,\ldots,\bH(\sigma_r) \br_r.
\end{equation*}
When using the quadrature rule~(\ref{eq:Hquadrature}) to compute these samples, one only needs to solve linear systems of the form $\bT(z_k)^\top \bb = \bell_j$ and $\bT(z_k)\bc = \br_j$ (to compute $\bb^\top = \bell_j^\top \bT(z_k)^{-1}$ and $\bc = \bT(z_k)^{-1} \br_j$).
For the computational procedure to succeed, some basic requirements regarding the number and selection of probing directions and sampling points must be satisfied \cite{brennan2023}. Typically we use probing directions that are sampled randomly from a standard Gaussian distribution, akin to \emph{sketching} in randomized numerical linear algebra. The full procedure is summarized in Algorithm~\ref{alg:multipointloewner}.

\begin{algorithm}
    \caption{Multipoint Loewner Algorithm for NLEVPs~\cite{brennan2023}; cf.~\cite{mayo2007}}
    \begin{algorithmic}[1]
        \Require{$\bT:\C \rightarrow \C^{n \times n}$, $\Omega \subset \C$, $(\theta_j,\bell_j), (\sigma_j,\br_j) \in \C \times \C^n$ for $j=1,\ldots,r$}
        \Ensure{$\bJ,\bV,\bW$}
        \State Approximate via quadrature, for $j=1,\ldots,r$,
        \begin{equation*}
            \bell_j^\top \bH(\theta_j) = \frac{1}{2 \pi \imath} \int_{\partial \Omega} \frac{1}{\theta_j - z} \bell_{j}^\top \bT(z)^{-1} \dop z, \ \ 
            \bH(\sigma_j) \br_j = \frac{1}{2 \pi \imath} \int_{\partial \Omega} \frac{1}{\sigma_j - z} \bT(z)^{-1}\br_j \; \dop z.
        \end{equation*}
        \State Form the Loewner matrices $\L, \L_s \in \C^{r \times r}$ via the entries
        \begin{equation} \label{eq:loewner_entries}
            (\L)_{ij} = \frac{\bell_i^\top [\bH(\theta_i) - \bH(\sigma_j)]@\br_j}{\theta_i - \sigma_j} \quad \text{and} \quad (\L_s)_{ij} = \frac{\bell_i^\top [\theta_i \bH(\theta_i) - \sigma_j \bH(\sigma_j)] @\br_j}{\theta_i - \sigma_j}.
        \end{equation}
        \State Compute the reduced SVDs with $\bSigma, \bSigma_s \in \R^{m \times m}$:
        \begin{equation}
        \label{eq:loewner_svds}
             \bX \bSigma \bY^* = \begin{bmatrix} \L & \L_s \end{bmatrix} \qquad \text{and} \qquad \bX_s \bSigma_s \bY_s^* = \begin{bmatrix} \L \\ \L_s \end{bmatrix}.
        \end{equation}
        \State Solve the generalized eigenvalue problem for  $\{\lambda_j\}_{j=1}^m$ and  $\{\bs_j\}_{j=1}^m$:
        \begin{equation*}
            \left(\bX^* \L_s \bY_s \right) \bs_j = \lambda_j \left(\bX^* \L \bY_s \right) \bs_j \quad \rightarrow \quad \bJ = \operatorname{diag}(\lambda_1,\ldots,\lambda_m).
        \end{equation*}
        \State Form the block matrices
        \begin{equation} \label{eq:loewner_BC}
            \bB^{\kern-1pt \top} = \begin{bmatrix}
                \bH(\theta_1)^{\kern-1pt \top} \bell_1 &
                \!\cdots\! &
                \bH(\theta_r)^{\kern-1pt \top} \bell_r
            \end{bmatrix}
            \ \text{and}\ \ \bC = \begin{bmatrix}
                \bH(\sigma_1) \br_1 & \!\cdots\! & \bH(\sigma_r) \br_r
            \end{bmatrix}.
        \end{equation}
        \State Use $\bS = [\bs_1,\ldots,\bs_m]\in\C^{m\times m}$ to compute eigenvectors of $\bT(z)$: 
        \begin{equation*}
            \bV = \bC@\bY_s \bS \quad \text{and} \quad \bW^* = -\bS^{-1} \left(\bX^* \L \bY_s \right)^{-1} \bX^* \bB.
        \end{equation*}
    \end{algorithmic}
    \label{alg:multipointloewner}
\end{algorithm}

While it might not be obvious in Algorithm~\ref{alg:multipointloewner}, rational interpolation plays a key role in the multipoint Loewner framework. Step~2 of the algorithm constructs the Loewner matrix $\L$ and shifted Loewner matrix $\L_s$ defined in~(\ref{eq:loewner_entries}), and step~5 builds the data matrices $\bB$ and $\bC$ in~(\ref{eq:loewner_BC}).  
To illustrate the rational interpolation aspect more clearly, let us assume that the Loewner matrices $\L$ and $\L_s$ in~(\ref{eq:loewner_entries}) are full rank. In this case, we can skip the SVD computation in~(\ref{eq:loewner_svds}) and the rational function 
\begin{equation*}
    \bG(z) = \bC (\L_s - z\L)^{-1} \bB
\end{equation*}
is a \emph{tangential interpolant} to $\bH(z)$, meaning that it satisfies
\begin{equation*}
    \bell_i^\top \bG(\theta_i) = \bell_i^\top \bH(\theta_i) \quad \text{and} \quad
    \bG(\sigma_j) \br_j = \bH(\sigma_j) \br_j,
\end{equation*}
under basic assumptions \cite{brennan2023}. A necessary condition for $\bG$ to exactly recover the rational function $\bH$ is that the rank of $\L$ coincides exactly with the number of eigenvalues of $\bT$ in $\Omega$. In this case (and still assuming that $\L$ and $\L_s$ are full rank, so $\operatorname{rank}(\L)=\operatorname{rank}(\L_s)=m$), we can write 
\begin{equation*}
    \bG(z) = \bC (\L_s - z\L)^{-1} \bB = \bV (z\bI - \bJ)^{-1} \bW^* = \bH(z).
\end{equation*}
Since we are interested in the exact recovery of the eigenvalues, we take the number of interpolation points $r$ large enough to ensure that $r \ge m = \operatorname{rank}(\L)=\operatorname{rank}(\L_s),$ where $m$ is the number of eigenvalues in $\Omega$ (counting multiplicity).

\begin{remark}
The Mayo--Antoulas Loewner framework is able to exactly recover a matrix-valued rational function. Hence, if the integrals are evaluated exactly rather than approximated via numerical quadrature and we operate in exact arithmetic, then the matrices $\bV,\bW$ and $\bJ$ from Theorem~\ref{theorem:keldysh} are recovered exactly. This claim even holds for the general decomposition in Theorem~\ref{theorem:generalKeldysh}, if the eigenvalue problem in step~4 of Algorithm~\ref{alg:multipointloewner} is replaced by the computation of a Jordan canonical form. Additionally, the eigenvector normalization stated in Theorem~\ref{theorem:generalKeldysh} is automatically satisfied as well. However, numerical determination of the Jordan form is notoriously challenging (see~\cite{Pet24} for a contemporary approach), and quadrature and rounding errors make an exact recovery of the Jordan structure of $\bH$ unlikely in practice.  
\end{remark}

\begin{remark}
Since the Mayo--Antoulas Loewner framework was introduced~\cite{mayo2007}, several equivalent formulations have appeared in the literature. Our presentation in Algorithm~\ref{alg:multipointloewner} follows the more recent formulations of \cite{ABG20,antoulas2017}, rather than the original work \cite{mayo2007} and the version used for NLEVPs in \cite{brennan2023}. While these methods are equivalent in theory, their numerical outcomes will differ in practice due to variations in the underlying implementations.

    The rational interpolation approach proposed in 1986 by Antoulas and Anderson~\cite{antoulas_scalar_1986} is also commonly called ``the Loewner framework.'' This method, which is designed to compute scalar-valued rational interpolants, differs fundamentally from the Mayo--Antoulas Loewner framework introduced in 2007, which computes matrix-valued rational functions based on tangential samples. The Antoulas--Anderson framework naturally extends to parametric systems via the parametric Loewner framework~\cite{ionita2014}, 
which plays an important role in the algorithmic developments of Sections~\ref{sec:parametricNLEVP} and \ref{sec:parametricMultipointLoewner}.
A parametric extension of the Mayo--Antoulas framework is not yet known.
\end{remark}

\section{Extending Keldysh's theorem to the parametric setting}
\label{sec:parametricKeldysh}
To adapt the ideas introduced in the previous section to solve pNLEVPs, we first seek to extend Theorem~\ref{theorem:keldysh} to the parametric setting. Toward this goal we begin by reviewing some prior work related to pNLEVPs. The theory for the linear case
\begin{equation*}
    \bT(z,p) = z\bI - \bA(p),
\end{equation*}
where $\bA$ is a matrix-valued analytic function, is well understood; classical eigenvalue perturbation theory describes the effects of asymptotically small changes in $p$~\cite{baumgartel1984,kato1995}.
Pole-residue forms for $\bT(z,p)^{-1}$ and the relation between the parameter $p$ and the eigenvalues and eigenvectors of $\bA$ have been thoroughly investigated. We will use these results to highlight certain special cases (e.g., situations where eigenvalues and eigenvectors have branch points or singularities), but will primarily focus our discussion on the case where $\bT(z,p)$ may have an arbitrary analytic dependence on~$z$. This more general setting has received far less attention, although some results, limited to the cases with only simple and semi\-simple eigenvalues, have been derived, e.g., in \cite{andrew1993}. Our forthcoming discussion thus reveals insights about the structure of pNLEVPs that, to our knowledge, have not been considered elsewhere.

\subsection{A pointwise decomposition}
We begin by introducing a preliminary parametric version of Theorem~\ref{theorem:keldysh}. The statements in this proposition do not give us the deeper insights we seek, regarding the nature of dependence on~$p$. 
However, we include this preliminary result to contrast it with the main theorem to come.
\begin{proposition}
    \label{proposition:pointwiseKeldysh}
    Suppose $\bT : \C\times\C\rightarrow \C^{n \times n}$ is a matrix-valued function such that, for each $p\in\cP\subset \C$, we have: (1)~$\bT(@\cdot@,p)$ is analytic in the domain $\Omega \subset \C$, and $\det(\bT(\cdot,p)) \not\equiv 0$; (2)~$\bT(@\cdot@,p)$ has $m$ simple eigenvalues in $\Omega$ denoted as $\lambda_1(p),\ldots,\lambda_m(p)$ with associated left and right eigenvectors forming the columns of $\bV(p) = [\bv_1(p),\ldots,\bv_m(p)] \in \C^{n \times m}$ and $\bW(p) = [\bw_1(p),\ldots,\bw_m(p)] \in \C^{n \times m}$. Then 
    \begin{equation}
    \label{eq:pointwiseKeldysh}
    \bT(z,p)^{-1} = \bV(p) (z \bI - \bJ(p))^{-1} \bW(p)^{*} + \bN(z,p),
    \end{equation}
    where $\bJ(p) = \operatorname{diag}(\lambda_1(p),\ldots,\lambda_m(p))$ and $\bN(\cdot,p)$ is analytic in $\Omega$. The eigenvectors are normalized such that $\bw_j(p)^* \bT_j(p) \bv_j(p) = 1$, where $\bT_j(p) = \frac{\partial}{\partial z} \bT(z,p) \Big\rvert_{z = \lambda_j(p)}$.
\end{proposition}
\begin{proof}
    The assumptions allow us to apply Theorem~\ref{theorem:keldysh} to $\bT(\cdot,p)$ for each $p \in \cP$, which directly yields the decomposition in \eqref{eq:pointwiseKeldysh} with the stated properties.
\end{proof}
We emphasize that this proposition is a \emph{pointwise} result, in the sense that the conventional Keldysh decomposition is simply applied at fixed parameter values $p \in \cP$.\ \ By appealing to the full version of the Keldysh decomposition in Theorem~\ref{theorem:generalKeldysh}, we could extend this result to handle cases where $\bT(\cdot,p)$ has multiple eigenvalues for some $p\in\cP$ (provided the sum of the algebraic multiplicity of the eigenvalues in $\Omega$ is constant for all $p\in\cP$).  Crucially, the proposition does not make any claims about key properties (such as continuity and smoothness) of the functions $\bJ$, $\bV$, and $\bW$.\ \ In general, $\lambda_j$, $\bv_j$ and $\bw_j$ may have singularities or branch points in $\cP$ \cite{andrew1993,kato1995} for parameters $p \in \cP$ where the algebraic multiplicity of $\lambda_j(p)$ exceeds one. This scenario is discussed in \cite{andrew1993} and illustrated in Example~\ref{example:simple}. Branch points present a challenge for many pNLEVP solvers (see, e.g., \cite{beyn2011,bindel2008,pradovera2024}), and special care must be taken to detect them accurately. We propose an algorithm that naturally handles these critical cases without requiring any additional computational steps or prior knowledge. The key ingredient will be a truly \emph{parameteric} extension of the Keldysh decomposition.

\begin{example}
    \label{example:simple}
    Consider a linear parametric eigenvalue problem with
    \begin{equation*}
        \bT(z,p) = z \left[\begin{array}{ccc}
            1 & 0 & 0 \\ 0 & 1 & 0 \\ 0 & 0 & 1
        \end{array}\right] - \left[\begin{array}{ccc}
            0 &\ 1 & \ 0 \\ 1-p &\ 0 & \ 0 \\ 0 &\ 1 & \ p
        \end{array}\right],
    \end{equation*}
    a parameter set $\cP = [0.75, 1.25]\subset\R$ and eigenvalue domain $\Omega = \left\{ z \in \C : |z| < 0.6 \right\}$. For this problem we have
    \begin{equation*}
        \det(\bT(z,p)) = (p-z)(1-p-z^2).
    \end{equation*}
    The eigenvalues are given by the zeros of the above polynomial:
    \begin{equation*}
        \lambda_1(p) = p, \quad \lambda_2(p) = +\sqrt{1-p}, \quad \text{and} \quad \lambda_3(p) = -\sqrt{1-p}.
    \end{equation*}
    For $p \in \cP$ only the eigenvalues $\lambda_2(p)$ and $\lambda_3(p)$ are inside $\Omega$, as shown in the left plot of Figure~\ref{figure:simpleKeldysh}. These eigenvalues have a branch point at $p=1$, for which $\bT$ has a defective eigenvalue with algebraic multiplicity $2$. For $p \in \cP \setminus \{ 1 \}$, the function $\bH$ of the Keldysh decomposition presented in Proposition~\ref{proposition:pointwiseKeldysh} can be written as
   \begin{equation*}
    \bH(z,p) = \bV(p)
    \Bigg( 
    	z \bI - 
    	\underbrace{
    	\left[\begin{array}{cc}
    		\sqrt{1 - p} & 0 \\[10pt]
    		0 & -\sqrt{1 - p}
    	\end{array}\right]
    	}_{\mbox{\small ${}=\bJ(p)$}}
    \Bigg)^{-1} \bW(p)^*,
    \end{equation*}
where 
   \begin{equation*} \footnotesize
    \bV(p) = \left[\!\begin{array}{cc}
    	\displaystyle{1 - \frac{p}{\sqrt{1 - p}}} & \displaystyle{\frac{p}{\sqrt{1 - p}} + 1} \\[16pt]
    	\sqrt{1 - p} - p & -p - \sqrt{1 - p} \\[12pt]
    	1 & 1
    \end{array}\!\right], \quad 
    \bW(p) = 
    \left[\!\begin{array}{cc}
    	\displaystyle{-\frac{(p + \sqrt{1 - p}) \sqrt{1 - p}}{2(p^2 + p - 1)}}& \displaystyle{-\frac{p + \sqrt{1 - p}}{2(p^2 + p - 1)}} \\[14pt]
    	\displaystyle{\phantom{-}\frac{(p - \sqrt{1 - p}) \sqrt{1 - p}}{2(p^2 + p - 1)}} & \displaystyle{-\frac{p - \sqrt{1 - p}}{2(p^2 + p - 1)}} \\[14pt] 0 & 0
    \end{array}\!\right],
    \end{equation*}
    while for $p = 1$ we have
    \begin{equation*}
        \bH(z,1) = \underbrace{\left[\begin{array}{cc}
                    1 & \phantom{-}0 \\
                    0 & \phantom{-}1 \\
                    0 & -1
                \end{array}\right]}_{\mbox{\small ${}=\bV(1)$}} \Bigg( z \bI - \underbrace{\left[\begin{array}{cc}
                    0 & 1 \\
                    0 & 0
                \end{array}\right]}_{\mbox{\small ${}=\bJ(1)$}}\Bigg)^{-1} {\underbrace{\left[\begin{array}{cc}
                1 & 0 \\
                0 & 1 \\
                0 & 0
            \end{array}\right]}_{\mbox{\small ${}=\bW(1)$}}}^*.
    \end{equation*}
    The functions $\bJ$, $\bV$ and $\bW$ all have a branch point at $p=1$. Note that $\bJ$ is \emph{discontinuous} in its top-right entry due to the appearance of a Jordan block in $\bJ(1)$; the functions $\bV$ and $\bW$ also have discontinuities at $p=1$.
\end{example}

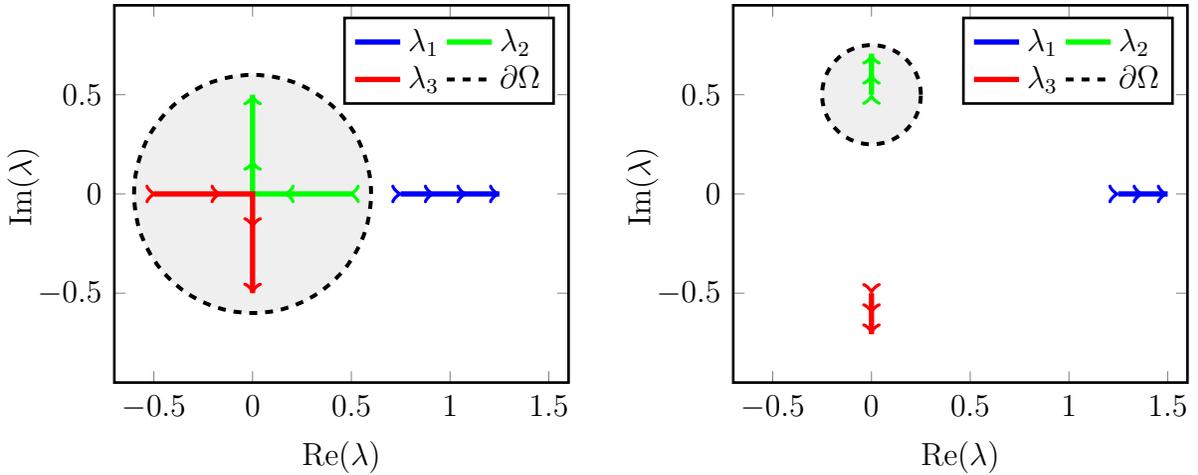
\begin{figure}[t!]
    \centering
    \begin{subfigure}{0.49\textwidth}
    \begin{tikzpicture}
      \begin{axis}[
        axis line style={line width=1pt},
        legend style={line width=1pt},
        width=3in,
        xmin=-0.7, xmax=1.6,
        ymin=-0.7, ymax=0.7,
        xlabel={$\Re(\lambda)$},
        ylabel={$\Im(\lambda)$},
        axis equal,
        line width=1.5pt,
        every axis plot/.style={line width=2pt},
        every axis y label/.style={at={(-0.2,.52)},rotate=90,anchor=center},
        legend columns=2,
        legend pos=north east,
        legend cell align=left,
        legend entries={
          $\lambda_1$,
          $\lambda_2$,
          $\lambda_3$,
          $\partial \Omega$
        }
      ]
        \addlegendimage{blue, line width=2pt}
        \addlegendimage{green, line width=2pt}
        \addlegendimage{red, line width=2pt}
        \addlegendimage{black, line width=1.5pt, dashed}
    
        \tikzset{
          arrow on path/.style={
            postaction={
              decorate,
              decoration={
                markings,
                mark=between positions 0.01 and 1 step 0.33 with {\arrow[line width=1.5pt]{>}}
              }
            }
          }
        }

        \filldraw[fill=gray!25, fill opacity=0.5, draw=none, dashed] (axis cs:0,0) circle [radius=0.6];
        \addplot[blue, arrow on path] file {figure_data/simple_l1.dat};
        \addplot[green, arrow on path] file {figure_data/simple_l2.dat};
        \addplot[red, arrow on path] file {figure_data/simple_l3.dat};
        \draw[line width=1.5pt, black, dashed] (axis cs:0,0) circle [radius=0.6];
    
      \end{axis}
    \end{tikzpicture}
    \end{subfigure}
    \begin{subfigure}{0.49\textwidth}
    \begin{tikzpicture}
      \begin{axis}[
        axis line style={line width=1pt},
        legend style={line width=1pt},
        width=3in,
        xmin=-0.7, xmax=1.6,
        ymin=-0.7, ymax=0.7,
        xlabel={$\Re(\lambda)$},
        ylabel={$\Im(\lambda)$},
        axis equal,
        line width=1.5pt,
        every axis plot/.style={line width=2pt},
        every axis y label/.style={at={(-0.2,.52)},rotate=90,anchor=center},
        legend columns=2,
        legend pos=north east,
        legend cell align=left,
        legend entries={
          $\lambda_1$,
          $\lambda_2$,
          $\lambda_3$,
          $\partial \Omega$
        }
      ]
        \addlegendimage{blue, line width=2pt}
        \addlegendimage{green, line width=2pt}
        \addlegendimage{red, line width=2pt}
        \addlegendimage{black, line width=1.5pt, dashed}
    
        \tikzset{
          arrow on path/.style={
            postaction={
              decorate,
              decoration={
                markings,
                mark=between positions 0.01 and 1 step 0.49 with {\arrow[line width=1.5pt]{>}}
              }
            }
          }
        }

        \filldraw[fill=gray!25, fill opacity=0.5, draw=none, dashed] (axis cs:0,0.5) circle [radius=0.25];
        \addplot[blue, arrow on path] file {figure_data/simple2_l1.dat};
        \addplot[green, arrow on path] file {figure_data/simple2_l2.dat};
        \addplot[red, arrow on path] file {figure_data/simple2_l3.dat};
        \draw[line width=1.5pt, black, dashed] (axis cs:0,0.5) circle [radius=0.25];
    
      \end{axis}
    \end{tikzpicture}
    \end{subfigure}

\vspace*{-7pt}
    \caption{Eigenvalues for Example~\ref{example:simple} and the first setup of Example~\ref{example:simpleKeldysh} with $\cP=[0.75,1.25]$ (left) and the second setup of Example~\ref{example:simpleKeldysh} with $\cP=[1.25,1.5]$ (right). The eigenvalues are shown as parametrized curves, where $p$ varies from the lower to the upper bound of $\cP$. In the first case (left) the eigenvalues $\lambda_2(p) = \sqrt{1-p}$ and $\lambda_3(p)=-\sqrt{1-p}$ coincide at the origin when $p=1$: $\lambda_2$ and $\lambda_3$ exhibit a branch point, giving nonsmooth behavior at $z=0$. However, the parametric Keldysh decomposition (Theorem~\ref{theorem:parametricKeldysh}) yields a \emph{rational function} $\bH$.\ \ The critical point at the origin is captured by the polynomial $u(z,p) = 1 - p - z^2$ in the denominator of $\bH$; this $u$ has the eigenvalues $\lambda_2$ and $\lambda_3$ as its zeros. In the second case (right plot) only one analytic branch is contained in $\Omega$, necessarily giving an $\bH$ with non-rational dependence on $p$.}
    \label{figure:simpleKeldysh}
\end{figure}

\subsection{A parametric Keldysh decomposition}
\label{sec:parametricKeldyshProperties} 
As our main theoretical result, we want to show that we can write
\begin{equation*}
    \bT(z,p)^{-1} = \bH(z,p) + \bN(z,p),
\end{equation*}
such that $\bH$ is a fraction of two analytic functions whose denominator has zeros that coincide exactly with the eigenvalues $\lambda_1(p),\ldots,\lambda_m(p)$ of $\bT(\cdot,p)$ in $\Omega$.\ \ It follows from the discussion in \cite{andrew1993} that this decomposition of $\bT(z,p)^{-1}$ holds if all eigenvalues are simple, but when multiplicities are present, it is far from clear why such an $\bH$ should exist. In fact, the presence of branch points in $\cP$ for the functions $\bJ$, $\bV$, and $\bW$ makes the following result interesting. This theorem is stated for parameters $p \in \pdom$, where $\pdom \subset \C$ is a domain, rather than for the previously introduced parameter set $\cP$. This choice is motivated by the need for $\bT$ to be analytic on a complex domain, which may not hold if $\cP$ is, for example, a real interval. Since one can take $\pdom$ to be a domain that barely includes $\cP$, this additional assumption is not overly restrictive.
\begin{theorem}[Parametric Keldysh]
    \label{theorem:parametricKeldysh}
    Suppose $\bT : \C\times \C \rightarrow \C^{n \times n}$ is analytic in the domain $\Omega \times \pdom$. For all $p \in \pdom$ assume $\bT(@\cdot@,p)$ has exactly $m$ eigenvalues (counting algebraic multiplicity) in the domain $\Omega$ and $\det(\bT(@\cdot@,p)) \not\equiv 0$. Then we can write
    \begin{equation}
        \label{eq:parametricKeldysh}
        \bT(z,p)^{-1} = \bH(z,p) + \bN(z,p),
    \end{equation}
    where, for all $p \in \pdom$, the function $\bN(@\cdot@,p)$ is analytic in $\Omega$, $\bH(@\cdot@,p)$ is a rational function, and we can write $\bH(z,p) = \bR(z,p) / u(z,p)$, where $\bR:\C\times\C\to \C^{n\times n}$ and $u:\C\times\C\to \C$ are analytic in $\Omega \times \pdom$. Further, for all $p \in \pdom$ the zeros of the degree-$m$ polynomial $u(@\cdot@,p)$ equal the $m$ eigenvalues of $\bT(@\cdot@,p)$ in $\Omega$, and we have
    \begin{equation*}
        \bH(z,p) = \bV(p) \left( z \bI - \bJ(p) \right)^{-1} \bW(p)^*,
    \end{equation*}
    where $\bJ(p)$, $\bV(p)$ and $\bW(p)$ are formed by the eigenvalues (Jordan blocks) and (generalized) eigenvectors, as in Theorem~\ref{theorem:generalKeldysh}. Finally, if for all $p \in \pdom$ all eigenvalues of $\bT(@\cdot@,p)$ in $\Omega$ are simple, then $\bJ$, $\bV$ and $\bW$ are analytic in $\pdom$.
\end{theorem}
We sketch the key steps of the proof of Theorem~\ref{theorem:parametricKeldysh}.
(The full proof and supporting background are provided in Appendix~\ref{sec:proofAppendix}.) The proof relies on two fundamental results from the theory of several complex variables: the Weierstrass preparation theorem and the Weierstrass division theorem \cite{ebeling2007,gunning2015}. 
First, we show that we can write
\begin{equation*}
    \bT(z,p)^{-1} = \frac{\bP(z,p)}{u(z,p)},
\end{equation*}
where $\bP(z,p)$ is analytic in $\Omega \times \Pi$ and $u(z,p)$ is a polynomial of degree $m$ in $z$ and analytic in $p$, such that the singularities (i.e., the eigenvalues) of $\bT$ in $\Omega \times \pdom$ coincide exactly with the zeros of $u$ in $\Omega \times \pdom$.\ \ This result follows from the identity
\begin{equation*}
    \bT(z,p)^{-1} = \frac{\operatorname{adj}(\bT(z,p))}{\det(\bT(z,p))},
\end{equation*}
the Weierstrass preparation theorem, and an analytic continuation argument. Similarly, the Weierstrass division theorem and analytic continuation allow us to write
\begin{equation*}
    \bP(z,p) = \bN(z,p)@@u(z,p) + \bR(z,p),
\end{equation*}
where $\bN$ and $\bR$ are analytic in $\Omega \times \pdom$ and $\bR(z,p)$ is a polynomial in $z$. Combined, these results give
\begin{equation*}
    \bT(z,p)^{-1} \;=\; \frac{\bN(z,p)@@u(z,p) + \bR(z,p)}{u(z,p)} \;=\; \bN(z,p) + \underbrace{\frac{\bR(z,p)}{u(z,p)}}_{{}\;=\;\bH(z,p)} \;=\; \bH(z,p) + \bN(z,p),
\end{equation*}
with the properties of $\bN$ and $\bH$ stated in Theorem~\ref{theorem:parametricKeldysh}. The relation to $\bJ(p)$,$\bV(p)$, and $\bW(p)$ follows from showing that, for each fixed $p$, the Laurent series of $\bH(z,p)$ in $z$ coincides exactly with that of the function $\bV(p)(z \bI - \bJ(p))^{-1} \bW(p)^*$ appearing in the pointwise parametric Keldysh decomposition in Proposition~\ref{proposition:pointwiseKeldysh}.

Theorem~\ref{theorem:parametricKeldysh} reveals how $\bH$ depends on $p$ even at the values of $p$ where the algebraic multiplicity of the individual eigenvalues of $\bT(@\cdot@,p)$ in $\Omega$ changes, its primary advantage over the pointwise result in Proposition~\ref{proposition:pointwiseKeldysh}.
We note that the assumption that $\bT$ has exactly $m$ eigenvalues in $\Omega$ (summing over all algebraic multiplicities) for every $p \in \pdom$ is crucial  to ensure that the dimensions of the matrices $\bJ(p)$, $\bV(p)$ and $\bW(p)$ are consistent and the underlying functions are well-defined on $p\in\pdom$.

In the non-parametric case, Algorithm~\ref{alg:multipointloewner} can exactly recover the eigenvalues in $\Omega$ via the Loewner framework because $\bH$ is a rational function. In the parametric setting, exact recovery of a scalar-valued rational function is also possible via the parametric Loewner frame\-work \cite{ionita2014}. However, in general the $\bH$ in Theorem~\ref{theorem:parametricKeldysh} is a rational function in the first variable $z$ but \emph{merely analytic} in the parameter $p$. An important question, if we aim for exact eigenvalue recovery, thus arises: When is $\bH$ a rational function in both $z$ and $p$? If so, we can (at least theoretically) recover parametric eigenvalues exactly, as we will discuss in the next two sections. The following lemma addresses this question.
\begin{lemma}
    Suppose $\bT : \C\times \C \rightarrow \C^{n \times n}$ satisfies the assumptions of Theorem~\ref{theorem:parametricKeldysh}. For the given $\Omega$ and $\pdom$, let $\bH$ be the function from the parametric Keldysh decomposition~\eqref{eq:parametricKeldysh}. If $\bH$ is a rational function in both variables $z$ and $p$, then there exists a bivariate polynomial $u : \C\times \C\rightarrow \C$ and an analytic function $h:\C\times \C \rightarrow \C$ such that
    \begin{equation*}
        \det(\bT(z,p)) = u(z,p)@h(z,p),
    \end{equation*}
    and the $m$ eigenvalues of $\bT(\cdot,p)$ in $\Omega$ equal the zeros of $u(\cdot,p)$ for all $p \in \pdom$. 
\end{lemma}
The above lemma gives a necessary condition for $\bH$ to be rational in both $z$ and $p$. Additional assumptions on $\bT$ are needed to establish a sufficient condition for $\bH$ to be rational. The statement from the lemma can easily be verified by following the proof in Appendix~\ref{sec:parametricKeldyshProof}, and by replacing the Weierstrass polynomial in \eqref{eq:detdecomposition} (whose coefficients are arbitrary analytic functions in $p$) with the bivariate polynomial $u$ introduced in the lemma. Example~\ref{example:simpleKeldysh} below shows how the same $\bT(z,p)$ can lead to rational $\bH$ and irrational $\bH$, depending on the choice of $\Omega$ and $\cP$.
 
\begin{example}
    \label{example:simpleKeldysh}
    Here we revisit the problem introduced in Example~\ref{example:simple}, and begin by investigating its parametric Keldysh decomposition for $\cP = [0.75,1.25]\subset\R$ and $\Omega = \left\{ z \in \C: \lvert z \rvert < 0.6 \right\}$. We have $\bT(z,p)^{-1} = \bH(z,p) + \bN(z,p)$, where
    \begin{equation}
    \label{eq:linearExample1Keldysh}
        \bH(z,p) = \frac{1}{z^2 + p - 1} \left[\begin{array}{ccc}
					z & 1 & 0 \\[6pt]
                    1 - p & z & 0 \\[6pt]
					\displaystyle{\frac{(p+z)(p-1)}{p^2+p-1}} & \displaystyle{\frac{-p@z+p-1}{p^2+p-1}} & 0
			\end{array}\right],
    \end{equation}
    and
    \begin{equation*}
       \bN(z,p) = \frac{1}{p-z} \left[\begin{array}{ccc}
					0 & 0 & 0 \\[6pt]
                    0 & 0 & 0 \\[6pt]
                    \displaystyle{\frac{p-1}{p^2+p-1}} & \displaystyle{\frac{-p}{p^2+p-1}} & -1
			\end{array}\right].
    \end{equation*}
   Note that $\bH$ and $\bN$ are both rational functions of $z$ and $p$, despite the fact that the eigenvalues $\pm\sqrt{1-p}$ (irrational in $p$) are contained in $\Omega$. This situation is possible for this $\cP$ and $\Omega$ because the eigenvalues in $\Omega$ for all $p \in \cP$ are captured by the zeros of the second polynomial factor of
    \begin{equation*}
        \det(\bT(z,p)) = (z-p)(1-p-z^2).
    \end{equation*}
    Further, the singularities of $\bN$ ($p = z, (-1\pm \sqrt{5})/2$) are not in $\Omega \times \cP$, and $\bN$ is clearly analytic in a neighborhood of that set. Neither $\bH$ nor $\bN$ have branch points in $\Omega \times \cP$.

   The functions $\bH$ and $\bN$ are not always rational for this $\bT(z,p)$, though. Suppose that instead $\cP = [1.25,1.5]$ and $\Omega = \{ z\in\C: \lvert z - 0.5@@\imath \rvert < 0.25 \}$. In this case we have
    \begin{equation*}
    \mbox{$\bH(z,p) = {}$} \footnotesize
    \frac{1}{z - \sqrt{1-p}}
    \left[\begin{array}{ccc}
    	\displaystyle{\frac{1}{2}} & \displaystyle{\frac{1}{2\sqrt{1 - p}}} 
        & 0 \\[12pt]
    	\displaystyle{\frac{\sqrt{1 - p}}{2}} 
        & \displaystyle{\frac{1}{2}} & 0 \\[12pt]
    	\displaystyle{\frac{-(p + \sqrt{1 - p})\sqrt{1 - p}}{2(p^2 + p - 1)}} 
        & \displaystyle{\frac{-(p + \sqrt{1 - p})}{2(p^2 + p - 1)}} 
        & 0
    \end{array}\right],
    \end{equation*}
    while $\bN(z,p)$ equals
    \begin{equation*}
    \footnotesize
    \left[\!\begin{array}{ccc}
    	\displaystyle{\frac{1}{2(z + \sqrt{1 - p})}}
        & \displaystyle{\frac{-1}{2z\sqrt{1 - p} - 2p + 2}} 
        & 0 \\[15pt]
    	\displaystyle{\frac{-(p + \sqrt{1 - p})(p - \sqrt{1 - p})\sqrt{1 - p}}{2(z + \sqrt{1 - p})(p^2 + p - 1)}}
        & \displaystyle{\frac{1}{2(z + \sqrt{1 - p})}}
        & 0 \\[15pt]
    	\displaystyle{\frac{p - 1}{(p - z)(p^2 + p - 1)} + \frac{(p - \sqrt{1 - p})\sqrt{1 - p}}{2(z + \sqrt{1 - p})(p^2 + p - 1)}}
        & \displaystyle{\frac{-(p + z)(p + \sqrt{1 - p})}{2(z + \sqrt{1 - p})(p - z)(p^2 + p - 1)}}
        & \displaystyle{\frac{1}{z-p}}
    \end{array}\!\right].
    \end{equation*}
    These functions are not rational, since $\Omega\times \pdom$ only captures the eigenvalue that is the positive branch of $\sqrt{1-p}$, one of the two roots of $u(z,p) = 1 - p - z^2$. The analyticity properties established in Theorem~\ref{theorem:parametricKeldysh} are still satisfied for $\bH$ and $\bN$ in a neighborhood of $\Omega \times \cP$, but the functions have branch points outside this neighborhood. 
    
    Figure~\ref{figure:simpleKeldysh} visualizes the eigenvalues for both these scenarios. 
\end{example}

\section{Motivation: Solving pNLEVPs in an ideal setting}
\label{sec:parametricNLEVP}
Our next goal is to motivate the computational procedure for solving pNLEVPs that we will describe in detail in the next section. Our algorithm is inspired by the multipoint Loewner method for non-parametric NLEVPs introduced in Algorithm~\ref{alg:multipointloewner}.  We can summarize that algorithm for non-parametric problems in three main steps.
\begin{enumerate}
    \item Compute probed samples of the form $\bell^\top \bH(z)$ and $\bH(z) \br$ via contour integration of $\bT(\cdot)^{-1}$, for a modest number of interpolation points $z$ and probing vectors $\bell$ and $\br$.
    \item Use the Mayo--Antoulas Loewner framework \cite{mayo2007} to \emph{exactly} recover the rational matrix function $\bH$ based on these probed samples.
    \item Extract from the realization of $\bH$ the eigenvalues of $\bT(z)$ in $\Omega$ and the corresponding eigenvectors, as outlined in \cite[sect.~5]{brennan2023}.
\end{enumerate}
For now, we assume that contour integrals are evaluated exactly, without any quadrature errors. We discuss how this procedure could be extended to the parametric setting. Clearly, we could simply apply the multipoint Loewner algorithm to $\bT(\cdot,\hp)$ for each desired parameter instance $\hp$. However, this approach could be computationally demanding in applications that seek eigenvalues for many parameter values. The sampling in step~1 would be prohibitively expensive if $\bT(z,p)$ is a large matrix, or many quadrature nodes are required to approximate the contour integral for $\bH(z,p)$. Hence, we propose a different approach that requires a limited number of evaluations of $\bT$ in an ``offline phase'' and enables quick eigenvalue computations for new parameter values in an ``online phase''. In an ideal world, a parametric multipoint Loewner method would follow these steps.
\begin{enumerate}
    \item (Offline) Compute probed samples of the form $\bell^\top \bH(z,p)$ and $\bH(z,p) \br$ via contour integration of $\bT(\cdot,p)^{-1}$, for a modest number of interpolation points $z$, parameter values $p$, and probing vectors $\bell$ and $\br$.
    \item (Offline) Use a parametric version of the Mayo--Antoulas Loewner framework to \emph{exactly} recover the rational matrix function $\bH$ from the probed samples.
    \item (Online) For each parameter instance $\hp$, evaluate the realization $\bH(\cdot, \hp)$ and extract from it the eigenvalues  of $\bT(z,\hp)$ in $\Omega$ and corresponding eigenvectors.
\end{enumerate}
There are two fundamental issues with this ``ideal parametric multipoint Loewner method.'' First, at present there is no known parametric version of the Mayo--Antoulas Loewner framework (or analogous algorithm) that \emph{exactly recovers} a \emph{matrix-valued} multivariate rational function from probed data. Second, the function $\bH$ may not be rational to begin with (as illustrated in Example~\ref{example:simpleKeldysh}), in which case  the exact recovery of $\bH$ with rational functions is not a feasible objective. For now, let us assume that $\bH$ is a rational function and focus on solving the former of the two issues. While there is no known algorithm that exactly recovers multivariate \emph{matrix-valued} rational functions from probed data, the parametric Loewner interpolation framework introduced in \cite{ionita2014} by Ionita and Antoulas allows for exact recovery of \emph{scalar-valued} rational multivariate functions, and, with some modifications, vector-valued rational functions as well~\cite[chap.~6]{ABG20}. As we explain below, we can use these computational frameworks for exact recovery of $\bell^\top \bH(z,p)$ and $\bH(z,p)@\br$, ultimately leading to exact eigenvalue recovery. We achieve this by splitting step~2 of the ``ideal parametric multipoint Loewner method'' into two parts.
\begin{enumerate}
    \item[2.1.] (Offline) For each left probing direction $\bell$ and right probing direction $\br$, use the parametric Loewner framework to compute vector-valued rational functions $\bL$ and $\bR$ that \emph{exactly} recover $\bL(z,p) := \bell^\top \bH(z,p)$ and $\bR(z,p) := \bH(z,p)@\br$ for fixed $\bell$ and $\br$.
    \item[2.2.] (Online) For each parameter instance $\hp$, evaluate the realizations $\bL(\cdot, \hp)$ and $\bR(\cdot, \hp)$ to obtain (inexpensive) samples of the form $\bell^\top \bH(z,\hp)$ and $\bH(z,\hp)@\br$ for interpolation points $z$.  Then apply the single-variable Mayo--Antoulas Loewner framework to this data for exact recovery of $\bH(\cdot,\hp)$.
\end{enumerate}
As we will discuss in the next section, in theory the above procedure allows for \emph{exact} recovery of eigenvalues and eigenvectors, even for arbitrary Jordan structures. This is of course only true with our previously mentioned assumptions, which include that $\bH$ is a rational function in both $z$ and $p$. While the latter is assumed implicitly via our assumptions in Theorem~\ref{theorem:parametricKeldysh}, the parameter dependence may not always be rational, as illustrated in Example~\ref{example:simpleKeldysh}. In this case, we propose to use \emph{rational approximation} to accurately capture the parameter dependence. In particular, we propose to employ vector-valued multivariate rational approximants $\bL(z,p) \approx \bell^\top \bH(z,p)$ and $\bR(z,p) \approx \bH(z,p)@ \br$. In general this approach will not result in exact recovery of eigenvalues, but only in approximations to them. Many algorithms exist that compute multivariate rational approximants, but we would like to use a method that exactly recovers $\bell^\top \bH(z,p)$ and $\bH(z,p)@\br$ if these functions are rational. In other words, our multivariate rational approximation algorithm of choice should be able to recover the parametric Loewner interpolant, and hence the eigenvalues, when exact recovery \emph{is} possible. Two such algorithms are the adaptive data-driven pROM construction~\cite{Ant25} and the parametric (multivariate) adaptive Antoulas--Anderson (p-AAA) algorithm~\cite{rodriguez2023} (inspired by the univariate adaptive Antoulas--Anderson algorithm~\cite{NST18}). In the following, we work with p-AAA.\ \  We refer the reader to \cite{ABG20,Ant25,rodriguez2023,ionita2014} for details about these algorithms and background on multivariate rational interpolation and approximation. 

\section{The parametric multipoint Loewner algorithm for pNLEVPs}
\label{sec:parametricMultipointLoewner}

In this section, we propose a method for solving pNLEVPs inspired by the multipoint Loewner algorithm for non-parametric NLEVPs in Algorithm~\ref{alg:multipointloewner}.  We will walk through each step of this method, which is summarized in Algorithm~\ref{alg:pmpl}. 

The algorithm begins with an offline phase. Given the target eigenvalue domain $\Omega$ and parameter domain $\pdom$, invoke Theorem~\ref{theorem:parametricKeldysh} to get a decomposition of the form~(\ref{eq:parametricKeldysh}):  $\bT(z,p)^{-1} = \bH(z,p) + \bN(z,p)$.
To approximate the probed samples of $\bH$, apply a quadrature rule (nodes $\{z_t\}$, weights $\{w_t\}$) to the contour integral formulations:
\begin{equation}
\label{eq:parametricQuadrature}
    \begin{aligned}
        \bell_k^\top \bH(s_i,p_j) = \frac{1}{2 \pi \imath} \int_{\partial \Omega} \frac{1}{s_i - z} \bell_k^\top \bT\left(z,p_j\right)^{-1} \, \dop z &\approx  \sum_{t=1}^N \frac{w_t}{s_i - z_t} \bell_k^\top \bT(z_t,p_j)^{-1};\\
        \bH(s_i,p_j) \br_k = \frac{1}{2 \pi \imath} \int_{\partial \Omega} \frac{1}{s_i - z} \bT\left(z,p_j\right)^{-1} \br_k \, \dop z &\approx  \sum_{t=1}^N \frac{w_t}{s_i - z_t} \bT(z_t,p_j)^{-1} \br_k.
    \end{aligned}
\end{equation}
Compute these quadrature approximations for all combinations of probing directions $\bell_k \in \{ \bell_1,\ldots,\bell_r \}$ and $\br_k \in \{ \br_1,\ldots,\br_r \}$, sampling values $s_i \in \{ s_1,\ldots, s_{2r} \}$, and parameter sampling points $p_j \in \{ p_1,\ldots,p_q \}$, leading to sets of approximated probed samples of the form $\bell_k^\top \bH(s_i, p_j)$ and $\bH(s_i, p_j)@\br_k$. (We assess the computational cost of this step and the following at the end of this section.) In the ideal setting outlined in the last section, we would use these probed samples to compute rational functions $\bL_k(z,p)$ and $\bR_k(z,p)$ that exactly recover $\bell_k^\top \bH(z,p)$ and $\bH(z,p)@ \br_k$ via the vector-valued version of the parametric Loewner framework~\cite[chap.~6]{ABG20}. In a practical setting where exact recovery may not be not possible, we apply the p-AAA algorithm \cite{rodriguez2023} to compute rational approximations $\bL_k(z,p) \approx \bell_k^\top \bH(z,p)$ and $\bR_k(z,p)\approx \bH(z,p)@ \br_k$ for $k=1,\ldots, r$.  In Appendix~\ref{sec:rationalApproximation} (Algorithm~\ref{alg:rationalApproximation}) we outline an efficient way to construct such approximations.  (The specific choice of the p-AAA algorithm is not critical, but is provided for completeness.)  We recall the essential assumption from Theorem~\ref{theorem:parametricKeldysh} that for each parameter $p \in \Pi$, $\bT(\cdot,p)$ has precisely $m$ eigenvalues (counting algebraic multiplicity) in $\Omega$.\ \  This assumption can be explicitly verified in Algorithm~\ref{alg:rationalApproximation} by checking the ranks of the Loewner matrices in \eqref{eq:eigenvalueAssumptionCheck}.

\begin{algorithm}[t]
    \caption{Parametric Multipoint Loewner Algorithm for pNLEVPs}
    \label{alg:pmpl}
    \begin{algorithmic}[1]
        \Require{$\bT:\C\times \C \rightarrow \C^{n \times n}$, domain $\Omega \subset \C$, probing directions $\{ \bell_k \}_{k=1}^r, \{ \br_k \}_{k=1}^r$, offline sampling values $\{s_i\}_{i=1}^{2r} = \{ \theta_i \}_{i=1}^r \cup \{ \sigma_i \}_{i=1}^r \subset \C\setminus \overline{\Omega}$, $\{p_j\}_{j=1}^q \subset \cP$, online evaluation parameter $\hp \in \cP$}
        \Ensure{Eigenvalue matrix $\bJ(\hp)\in\C^{m\times m}$ and eigenvector matrices $\bV(\hp),\bW(\hp)\in\C^{n\times m}$}
        \medskip
        \StartOfflinePhase
        \State Use quadrature to approximate the contour integrals for all combinations of sampling values and probing directions ($i=1,\ldots,2r$, $j=1,\ldots,q$, and $k=1,\ldots,r$):
        \begin{equation*}
        \begin{aligned}
            \bell_k^\top \bH\left(s_i, p_j\right) &= \frac{1}{2 \pi \imath} \int_{\partial \Omega} \frac{1}{s_i - z} \bell_k^\top \bT\left(z, p_j\right)^{-1} \; \dop z  \\
            \bH\left(s_i, p_j\right) \br_k &= \frac{1}{2 \pi \imath} \int_{\partial \Omega} \frac{1}{s_i - z} \bT\left(z, p_j\right)^{-1} \br_k \; \dop z.
        \end{aligned}
        \end{equation*}
        \State Use rational approximation as in Algorithm~\ref{alg:rationalApproximation} or \cite{ABG20,Ant25,ionita2014} to exactly recover or approximate
        \begin{equation}
        \label{eq:probedApproximations}
            \bL_k(z,p) \approx \bell_k^\top \bH(z,p) \quad \text{and} \quad \bR_k(z,p) \approx \bH(z,p) \br_k.
        \end{equation}
        \StartOnlinePhase
        \State Form Loewner matrices $\L, \L_s \in \C^{r \times r}$ via the entries
        \begin{equation*}
            (\L)_{ij} = \frac{\bL_i(\theta_i,\hp)^{\kern-1pt \top} \br_j - \bell_i^\top \bR_j(\sigma_j,\hp)}{\theta_i - \sigma_j}, \qquad (\L_{s})_{ij} = \frac{\theta_i \bL_i(\theta_i,\hp)^{\kern-1pt \top} \br_j - \sigma_j \bell_i^\top \bR_j(\sigma_j,\hp)}{\theta_i - \sigma_j}.
        \end{equation*}
        \State Compute the reduced SVDs with $\bSigma, \bSigma_s \in \R^{m \times m}$:
        \begin{equation*}
             \bX \bSigma \bY^* = \begin{bmatrix} \L & \L_s \end{bmatrix} \qquad \text{and} \qquad \bX_s^{} \bSigma_s^{} \bY_s^* = \begin{bmatrix} \L \\ \L_s \end{bmatrix}.
        \end{equation*}
        \State Solve the generalized eigenvalue problem
        \begin{equation}
        \label{eq:JComputation}
            \left(\bX^* \L_s \bY_s \right) \bs_j = \lambda_j \left(\bX^* \L \bY_s \right) \bs_j \quad \rightarrow \quad \bJ = \operatorname{diag}(\lambda_1,\ldots,\lambda_m).
        \end{equation}
        \State Form the block matrices
        \begin{equation*}
            \bB^{\kern-1pt \top} = \begin{bmatrix}
                \bL_1(\theta_1,\hp) &
                \!\cdots\! &
                \bL_r(\theta_r,\hp)
            \end{bmatrix}
            \quad \text{and} \quad \bC = \begin{bmatrix}
                \bR_1(\sigma_1,\hp) & \!\cdots\! & \bR_r(\sigma_r,\hp)
            \end{bmatrix}.
        \end{equation*}
        \State Using $\bS = [\bs_1,\ldots,\bs_m]$, construct eigenvectors of $\bT(\cdot, \hp)$:
        \begin{equation*}
            \bV = \bC \bY_s \bS \quad \text{and} \quad \bW^* = -\bS^{-1} \left(\bX^* \L \bY_s \right)^{-1} \bX^* \bB.
        \end{equation*}
    \end{algorithmic}
\end{algorithm}

The online phase of Algorithm~\ref{alg:pmpl} develops approximations of $\bV(\hp), \bW(\hp)$ and $\bJ(\hp)$ from Theorem~\ref{theorem:parametricKeldysh} for an arbitrary $\hp \in \cP$. We obtain these quantities by following steps similar to lines~2--6 of Algorithm~\ref{alg:multipointloewner}. To make the relation between Algorithm~\ref{alg:multipointloewner} and the online phase of Algorithm~\ref{alg:pmpl} clear, compare lines~2--6 of the former with lines~3--7 of the latter. These lines are \emph{identical}, except that $\bell_i^\top \bH(\theta_i)$ is replaced by $\bL_i(\theta_i,\hp)$ and $\bH(\sigma_j)@\br_j$ is replaced with $\bR_j(\sigma_j,\hp)$. Hence, if $\bL_k(z,p) = \smash{\bell_k^\top} \bH(z,p)$ and $\bR_k(z,p) = \bH(z,p)@\br_k$ in \eqref{eq:probedApproximations}, then all the theoretical guarantees of the multipoint Loewner algorithm \cite{brennan2023} apply here as well. 
The next proposition, a natural consequence of the observation that lines~2--6 of Algorithm~\ref{alg:multipointloewner} essentially coincide with lines~3--7 of Algorithm~\ref{alg:pmpl}, makes this statement concrete.

\begin{proposition}
    In Algorithm~\ref{alg:pmpl}, assume that for all $z\in\Omega$ and $p\in\pdom$,
    \begin{equation*}
        \bL_k(z,p) = \bell_k^\top \bH(z,p) \quad \text{and} \quad \bR_k(z,p) = \bH(z,p) \br_k;
    \end{equation*}
    cf.~equation~\eqref{eq:probedApproximations}. Then the output of Algorithm~\ref{alg:pmpl} is equivalent to the output of Algorithm~\ref{alg:multipointloewner} applied to $\bT(\cdot,\hp)$ (assuming exact arithmetic and exact contour integral evaluations). If the assumptions for exact eigenvalue recovery of Algorithm~\ref{alg:multipointloewner} are satisfied \cite{brennan2023} and \eqref{eq:JComputation} is replaced with the computation of a Jordan canonical form, then Algorithm~\ref{alg:pmpl} recovers the matrices $\bJ(\hp)$, $\bV(\hp)$ and $\bW(\hp)$ from Theorem~\ref{theorem:parametricKeldysh}.
\end{proposition}

We emphasize that the offline and online phases of Algorithm~\ref{alg:pmpl} are distinct: after executing the offline phase \emph{just once}, the online phase can be executed many times for arbitrary $\hp \in \cP$.\ \ This element is crucial, as the offline phase can be expensive due to the need to solve linear systems involving $\bT(z_k,p_j)$ for the quadrature. The online phase is generally much faster, avoiding computation with~$\bT$.

\begin{remark}
    Algorithm~\ref{alg:pmpl} can be simplified significantly if we only seek eigenvalues (not eigenvectors), and all those eigenvalues are simple. In this case, choose random vectors $\bell$, $\br$, use quadrature to approximate
    \begin{equation*}
        \bell^\top \bH(s_i,p_j)@@\br = \frac{1}{2 \pi \imath} \int_{\partial \Omega} \frac{1}{s_i - z} \bell^\top \bT\left(z, p_j\right)^{-1} \br \; \dop z,
    \end{equation*}
    and then recover (or approximate) the scalar-valued function $f(z,p) = \bell^\top \bH(z,p)@@ \br$ via multivariate rational interpolation methods. In this case, the singularities of $f$ serve to approximate the singularities of $\bH$, i.e., the eigenvalues of $\bT$. 

    In general, the singularities of $\bell^\top \bH(z,p)@\br$, $\bL_k(z,p)$ or $\bR_k(z,p)$ in Algorithm~\ref{alg:pmpl} need not give us complete information about the eigenvalues of $\bT$.\ \ Consider
    \begin{equation*}
        \bH(z,p) = \frac{1}{z-p} \bI: \C\times \C \to \C^{n\times n}
    \end{equation*}
    for which $\lambda=p$ has algebraic multiplicity $n$.
    For any $\bell_k$ we have that $\bell_k^\top \bH(z,p)$ can be exactly recovered by a function that reads $\bL_k(z,p) = (z-p)^{-1}\bell_k^\top$. This scenario clearly shows that information about the multiplicity of eigenvalues is lost if we only observe the singularities of the functions $\bL_k$ or $\bR_k$. While the algebraic multiplicity and the eigenvectors could be obtained by computing (the dimension of) the nullspace of $\bT(\lambda_j(p),p)$, the precise Jordan structure and normalization of eigenvectors as in \eqref{eq:eigenvectorNormalization} is not recovered this way. In contrast, Algorithm~\ref{alg:pmpl} can in theory recover these structures if the required assumptions are satisfied.
\end{remark}

\begin{remark}
    In practice, the inputs to Algorithm~\ref{alg:pmpl} are more flexible than presented here. The number of probing directions may be chosen independently of the number of sampling values, and $\{ \theta_j \}_{j=1}^r$ and $\{ \sigma_j \}_{j=1}^r$ used in the online phase can be chosen independently of the sampling values $\{ s_j \}_{j=1}^{2r}$ that were used during the offline phase. This flexibility enables experimentation with different choices of sampling points during the online phase to achieve better numerical properties of the eigenvalue problem in \eqref{eq:JComputation}. These aspects of the multipoint Loewner method are discussed with more context in \cite{brennan2023}. Further, the online phase of Algorithm~\ref{alg:pmpl} is not necessarily restricted to $\hp \in \pdom$.\ \ In particular, we may choose any $\hp \in \C$ and compute eigenvalue approximations for $\bT(\cdot,\hp)$. In some cases this yields good approximations (as demonstrated in the numerical example in Section~\ref{sec:delayExample}), but may also give unreliable predictions (for example, as in the second setup in Section~\ref{sec:linearExample} for $\hp=1$). 
\end{remark}

    Let us consider the computational complexity of Algorithm~\ref{alg:pmpl}. Line~1 of the algorithm requires computing quadrature-based approximations of $\bell^\top \bH(z,p)$ and $\bH(z,p)@\br$ for all combinations of sampling points and probing directions. We state the computational complexity of this step in terms of linear systems of equations involving the matrix $\bT(z,p)$ that need to be solved. Suppose we use $N$ quadrature nodes for a contour integral approximation along the lines of \eqref{eq:parametricQuadrature}, $r$ left and right probing directions, $q$ parameter sampling values $\{ p_j \}_{j=1}^q$, and $d = 2r$ sampling values $\{ s_j \}_{j=1}^d$. Then line~1 of Algorithm~\ref{alg:pmpl} can be executed by solving $2 N q$ linear systems (each having $r$ right-hand sides) of the form $\bT(z,p)^\top \left[\bb_1,\ldots,\bb_r\right] = \left[\bell_1,\ldots,\bell_r\right]$ and $\bT(z,p)\left[\bc_1,\ldots,\bc_r\right] = \left[\br_1,\ldots,\br_r \right]$ (to compute $\bb_k^\top = \bell_k^\top \bT(z,p)^{-1}$ and $\bc_k = \bT(z,p)^{-1}\br_k$ for $k=1,\ldots,r$). This calculation requires evaluating $\bT$ at $Nq$ values. The rational interpolation procedure in 
line~2 (based on Algorithm~\ref{alg:rationalApproximation}) requires $\cO(q d^3 + d q c^2 m^2)$ floating-point operations, where $m+1$ and $c+1$ are the number of interpolated sampling points of the rational interpolant in $z$ and $p$, respectively. Executing the online phase of the algorithm is associated with $\cO(d^3)$ floating-point operations for a single parameter value $\hp$. Importantly, the cost for the online phase is independent of the quadrature nodes $N$ and the function $\bT$ and hence only involves the dimension $n$ for construction of the eigenvector matrices (if desired).
    
   In contrast, Algorithm~\ref{alg:multipointloewner} only required solving $2N$ linear systems, each having $r$ right-hand sides. The increased computational cost of Algorithm~\ref{alg:pmpl} is explained by the need to sample at each of the $q$ parameter sampling points. However, there is an additional discrepancy between the two algorithms worth noting. Algorithm~\ref{alg:multipointloewner} is based on the Mayo--Antoulas Loewner framework, in which each probing direction is associated with exactly one sampling point. In the parametric Loewner framework, probing directions and sampling points are independent of one another, and sampling is required for all combinations of them. This discrepancy leads to additional differences in both the numerical setup and the computational cost between the parametric and non-parametric algorithms arising in the rational interpolation step.

\section{Numerical examples}
\label{sec:numericalExamples}
In this section we demonstrate our previous theoretical discussion and algorithmic development via several numerical experiments. There are numerous inputs to Algorithm~\ref{alg:pmpl} and Algorithm~\ref{alg:rationalApproximation} whose effect on approximation quality we could study in detail, such as the choice of quadrature rule, number of quadrature nodes, number of sampling points, rational approximation convergence criterion, etc. Our primary goal here is not to investigate good choices for these inputs in detail, but rather to highlight the implications of the theoretical results from Section~\ref{sec:parametricKeldysh} for the computational setting. For algorithmic details pertaining to the multipoint Loewner method in the non-parametric case (how the number of quadrature nodes and location of sampling points affects accuracy and numerical stability), see~\cite{brennan2023}.  Practical and theoretical aspects of multivariate rational interpolation methods are discussed in \cite{ABG20,Ant25,rodriguez2023,ionita2014}. The results in this section were primarily obtained on an HP Laptop 15-dw1xxx, equipped with an Intel(R) Core(TM) i3-10110U CPU and 16 GB of RAM, using MATLAB R2024b Update 5 (24.2.0.2863752). The code is made available under an open-source license \cite{balicki2025code}.

\subsection{A linear parametric eigenvalue problem}
\label{sec:linearExample}

We begin be revisiting the linear parametric eigenvalue problem introduced in Example~\ref{example:simple}, to discuss the numerical aspects of the two setups introduced in Example~\ref{example:simpleKeldysh} and shown in Figure~\ref{figure:linearExample}. For both scenarios we run Algorithm~\ref{alg:pmpl} with $40$ uniformly spaced sampling values for $\{ s_i \}_{i=1}^{40}$ and $\{ p_j \}_{j=1}^{40}$. Further, $\{ \theta_i\}_{i=1}^{20}$ and $\{\sigma_i\}_{i=1}^{20}$ are selected as interlaced points from $\{ s_i \}_{i=1}^{40}$ to promote numerical stability of the Loewner matrices~\cite{EI22}. Additionally, we use $N=512$ nodes for a trapezoidal quadrature rule to approximate the contour integrals as in \eqref{eq:parametricQuadrature}.  Figure~\ref{figure:linearExample} shows the relative residuals for the computed eigenvalues.
For the first setup with $\cP = [0.75,1.25]$, the fact that $\bH(z,p)$ is rational in $z$ and $p$ makes us expect exact recovery of eigenvalues and eigenvectors. In this case, the rational approximation procedure from Appendix~\ref{sec:rationalApproximation} computes $\bL_k$ and $\bR_k$ with numerator and denominator polynomial degrees matching those of $\bell_k^\top \bH(z,p)$ and $\bH(z,p)@\br_k$. Hence, these functions are in theory recovered exactly by the Loewner interpolant. Of course, numerical errors lead to near-exact recovery in practice.

\begin{figure}[b!]
  \centering
  \begin{subfigure}[t]{0.49\textwidth}
    \centering
    \begin{tikzpicture}
      \begin{axis}[
        axis line style={line width=0.75pt},
        legend style={line width=0.75pt},
        width=3in,
        height=2in,
        grid=major,
        grid style={line width=.2pt, draw=gray!40, dashed},
        xmin = 0, xmax = 2,
        ymin = 1e-17, ymax = 1e3,
        ymode = log,
        ytick={1e-16,1e-10,1e-4,1e2},
        line width = 2pt,
        xlabel = {$p$},
        ylabel = {Relative Residual},
        legend style={at={(0.985,0.97)}, anchor=north east},
        legend cell align={left}
      ]

        \addplot[
            forget plot,
            draw=gray,
            line width=0.9pt,
            fill=gray!80,
            fill opacity=0.3,
            pattern=north east lines,
            pattern color=gray!80
        ] coordinates {
            (0.75,1e-17)
            (0.75,1e3)
            (1.25,1e3)
            (1.25,1e-17)
        } -- cycle;
      
        \addplot[red] table {figure_data/linear_example_1_left.dat};
        \addlegendentry{$p \notin \cP$}
        \addplot[blue] table {figure_data/linear_example_1_center.dat};
        \addlegendentry{$p \in \cP$}
        \addplot[red] table {figure_data/linear_example_1_right.dat};
      \end{axis}
    \end{tikzpicture}
  \end{subfigure}
  \begin{subfigure}[t]{0.49\textwidth}
    \centering
    \begin{tikzpicture}
      \begin{axis}[
        axis line style={line width=0.75pt},
        legend style={line width=0.75pt},
        width=3in,
        height=2in,
        grid=major,
        grid style={line width=.2pt, draw=gray!40, dashed},
        xmin = 0, xmax = 2,
        ymin = 1e-17, ymax = 1e3,
        ymode = log,
        ytick={1e-16,1e-10,1e-4,1e2},
        line width = 2pt,
        xlabel = {$p$},
        ylabel = {Relative Residual},
        legend style={at={(0.03,0.06)}, anchor=south west},
        legend cell align={left}
      ]

        \addplot[
            forget plot,
            draw=gray,
            line width=0.9pt,
            fill=gray!80,
            fill opacity=0.3,
            pattern=north east lines,
            pattern color=gray!80
        ] coordinates {
            (1.25,1e-17)
            (1.25,1e3)
            (1.5,1e3)
            (1.5,1e-17)
        } -- cycle;
      
        \addplot[red] table {figure_data/linear_example_2_left.dat};
        \addlegendentry{$p \notin \cP$}
        \addplot[blue] table {figure_data/linear_example_2_center.dat};
        \addlegendentry{$p \in \cP$}
        \addplot[red] table {figure_data/linear_example_2_right.dat};
      \end{axis}
    \end{tikzpicture}
  \end{subfigure}

\vspace*{-5pt}
  \caption{Maximum relative residuals $\lVert \bT(\lambda(p),p) \bv(p) \rVert_2 / \lVert \bv(p) \rVert_2$ for the two setups discussed in Section~\ref{sec:linearExample}. In the left plot we exactly recover eigenvalues and eigenvectors, up to numerical errors. (The peak at $p\approx 0.62\not\in\cP$ occurs because $\bH$ has a singularity close to that value.) In the right plot we seek an approximation of a single eigenvalue. The approximation quality deteriorates quickly as $p$ moves toward the branch point at $p=1 \notin \cP$. For $p>1.5$ no branch points appear, resulting in a good approximation of the eigenvalue and eigenvector. }
  \label{figure:linearExample}
\end{figure}
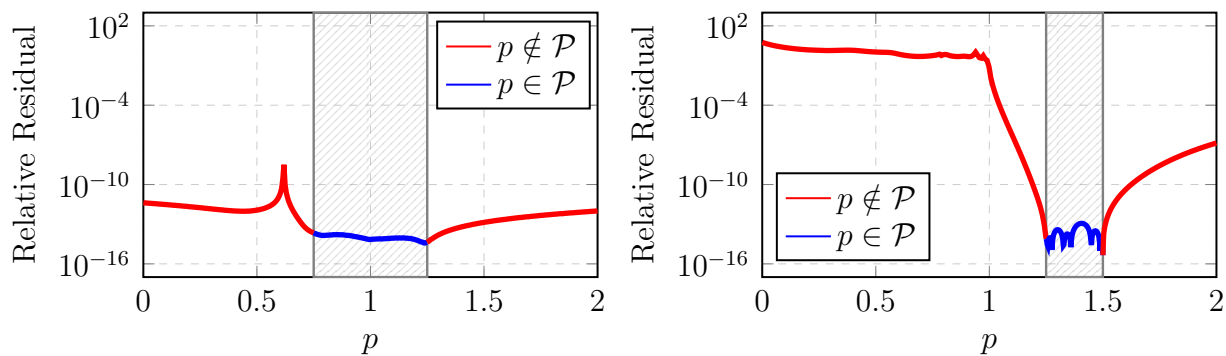

In the second setup where $\cP = [1.25,1.5]$, we capture one eigenvalue that corresponds to a single branch of a multivalued analytic function. In this case, $\bH(z,p)$ is not rational in $p$, and hence we merely approximate $\bell_k^\top \bH(z,p)$ and $\bH(z,p) \br_k$ via rational functions. In particular, Algorithm~\ref{alg:rationalApproximation} computes $\bL_k(z,p)$ and $\bR_k(z,p)$ with a denominator that is degree-$1$ in $z$ and degree-$7$ in $p$. For $p\in\cP$, this approximation is sufficient to estimate the  eigenvalue in $\Omega$ with errors close to machine precision for $p\in\cP$.\ \  For $p \notin \cP$ the approximation quality remains acceptable, except when $p$ moves toward and past the branch point at $p=1$.

In both these cases, the accuracy of the eigenpairs extend well beyond the limits of $\cP$, and indeed include parameter values for which the eigenvalues fall outside $\Omega$.\ \ It would be interesting to study this \emph{overconvergence} phenomenon more thoroughly.

\subsection{Stability analysis of a delay differential equation}
\label{sec:delayExample}

Next, we consider a pNLEVP that arises in the study of the delay differential equation
\[ \bx'(t) = -\bE@\bx(t) - 0.01 @ \bx(t-p).\] 
The parameter $p>0$ specifies the delay.
To assess the asymptotic stability of solutions, one seeks the rightmost eigenvalues of the pNLEVP
\begin{equation} \label{eqn:delay_pNLEVP}
    \bT(z,p) = \big(z + 0.01@\eop^{-p@z}\big) @\bI + \bE.
\end{equation}
If all eigenvalues are in the open left-half of the complex plane, then all solutions will be asymptotically stable.
(See~\cite{MN14} for details of the stability analysis.)
We take $\bE \in \R^{10 \times 10}$ to be diagonal with entries logarithmically spaced in $[10^{-4},10^{10}]$. 
Though $\bT$ is finite dimensional, it has infinitely many eigenvalues; our choice of $\bT$ allows these eigenvalues to be computed by evaluating branches of the Lambert-W function~\cite[p.~148]{MN14}. 
Figure~\ref{fig:plot_eig_p30} shows a portion of the spectrum for delay $p=30$ (computed in Mathematica).

\begin{figure}[b!]
\begin{center}
\includegraphics[width=3in]{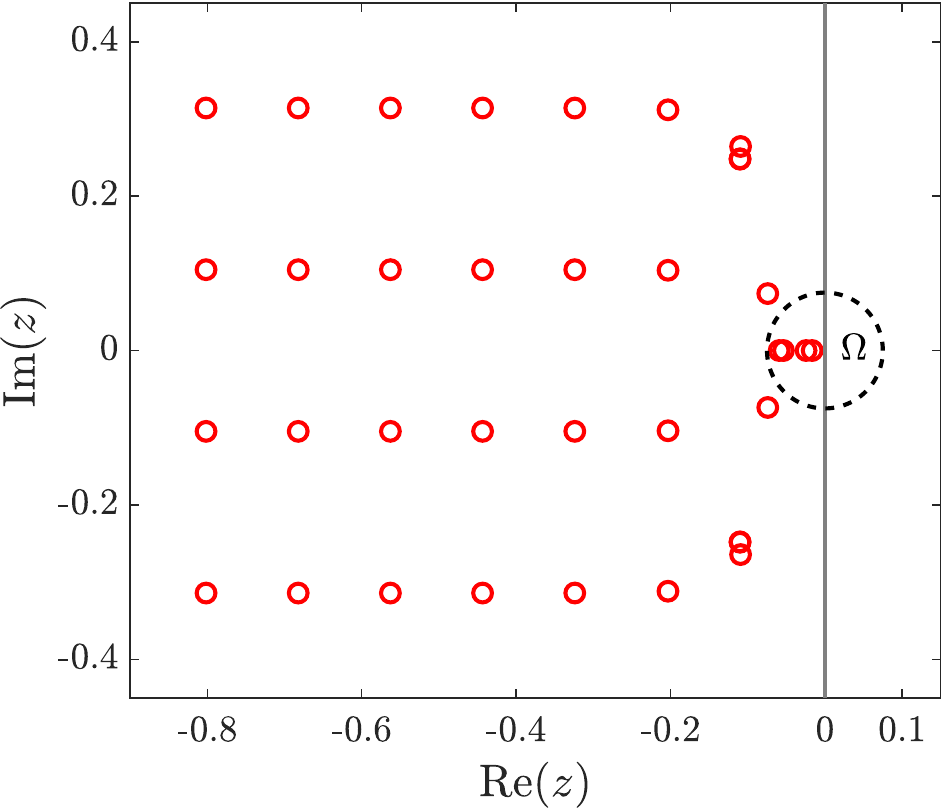}
\end{center}

\vspace*{-5pt}
\caption{\label{fig:plot_eig_p30}A portion of the spectrum for the pNLEVP~(\ref{eqn:delay_pNLEVP}) with parameter value $p=30$.  The rightmost eigenvalue determines the asymptotic stability of the system.}
\end{figure}

We seek to analyze how the asymptotic stability of solutions depends on the delay parameter $p$.  We pick a parameter set $\cP = [30,35]$ and domain $\Omega = \{ z \in \C \;:\; \lvert z \rvert < 0.075 \}$ of interest. As the inputs for Algorithm~\ref{alg:multipointloewner} we choose $\{s_i\}_{i=1}^{40}$ and $\{ p_j \}_{j=1}^{40}$ as $40$ uniformly sampled points in $\{ z \in \C: |z| = 0.1 \}$ and $\cP$, respectively. Additionally, $\{ \theta_i \}_{i=1}^{20}$ and $\{ \sigma_i \}_{i=1}^{20}$ are selected as interlaced points from $\{ s_i \}_{i=1}^{40}$ and $N=128$ quadrature nodes are used for contour integral approximations. The rational approximation procedure in Algorithm~\ref{alg:rationalApproximation} computes $\bL_k(z,p) \approx \bell_k^\top\bH(z,p)$ and $\bR_k(z,p) \approx \bH(z,p)@\br_k$ with numerator and denominator that are degree-$4$ polynomials in $z$ and degree-$5$ polynomials in $p$. In particular, this means that exactly four eigenvalues in $\Omega \times \cP$ are approximated through our computational procedure.

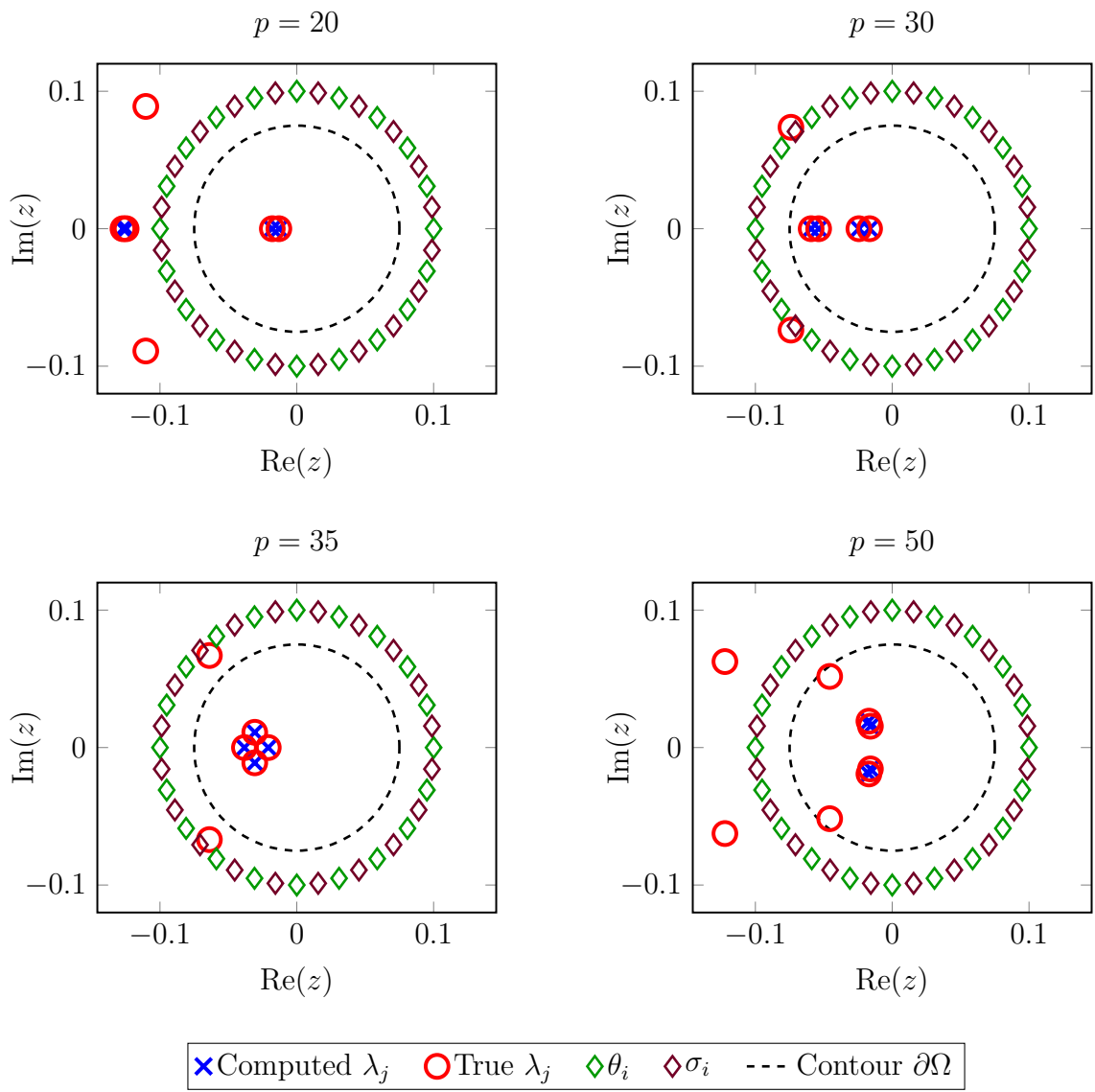
\begin{figure}[t]
\begin{subfigure}[t]{0.45\textwidth}
  \centering
  \begin{tikzpicture}
    \begin{axis}[
      title={$p = 20$},
      xlabel={$\Re(z)$},
      ylabel={$\Im(z)$},
      xtick={-0.1,0,0.1},
      ytick={-0.1,0,0.1},
      every axis y label/.style={at={(-0.18,.5)},rotate=90,anchor=center},
      xmin=-0.12, xmax=0.12,
      ymin=-0.12, ymax=0.12,
      line width=0.75pt,
      width=2.8in,
      axis equal
    ]
      \addplot[only marks, mark=x, blue, line width=1.5pt, mark size=4.5pt] table {figure_data/parametric_delay_20.dat};
      \addplot[only marks, mark=o, red, line width=1.5pt, mark size=4.5pt] table {figure_data/parametric_delay_reference_20.dat};
      \addplot[only marks, mark=diamond, green!60!black, line width=1pt, mark size=4pt] table {figure_data/parametric_delay_left_interpolation.dat};
      \addplot[only marks, mark=diamond, purple!60!black, line width=1pt, mark size=4pt] table {figure_data/parametric_delay_right_interpolation.dat};
      \draw[line width=1pt, black, dashed] (axis cs:0,0) circle [radius=0.075];
    \end{axis}
  \end{tikzpicture}
\end{subfigure}%
\qquad
\begin{subfigure}[t]{0.45\textwidth}
  \centering
  \begin{tikzpicture}
    \begin{axis}[
      title={$p = 30$},
      xlabel={$\Re(z)$},
      ylabel={$\Im(z)$},
      xtick={-0.1,0,0.1},
      ytick={-0.1,0,0.1},
      every axis y label/.style={at={(-0.18,.5)},rotate=90,anchor=center},
      xmin=-0.12, xmax=0.12,
      ymin=-0.12, ymax=0.12,
      line width=.75pt,
      width=2.8in,
      axis equal
    ]
      \addplot[only marks, mark=x, blue, line width=1.5pt, mark size=4.5pt] table {figure_data/parametric_delay_30.dat};
      \addplot[only marks, mark=o, red, line width=1.5pt, mark size=4.5pt] table {figure_data/parametric_delay_reference_30.dat};
      \addplot[only marks, mark=diamond, green!60!black, line width=1pt, mark size=4pt] table {figure_data/parametric_delay_left_interpolation.dat};
      \addplot[only marks, mark=diamond, purple!60!black, line width=1pt, mark size=4pt] table {figure_data/parametric_delay_right_interpolation.dat};
      \draw[line width=1pt, black, dashed] (axis cs:0,0) circle [radius=0.075];
    \end{axis}
  \end{tikzpicture}
\end{subfigure}

\vspace{1em}

\begin{subfigure}[t]{0.45\textwidth}
  \centering
  \begin{tikzpicture}
    \begin{axis}[
      title={$p = 35$},
      xlabel={$\Re(z)$},
      ylabel={$\Im(z)$},
      xtick={-0.1,0,0.1},
      ytick={-0.1,0,0.1},
      every axis y label/.style={at={(-0.18,.5)},rotate=90,anchor=center},
      xmin=-0.12, xmax=0.12,
      ymin=-0.12, ymax=0.12,
      line width=.75pt,
      width=2.8in,
      axis equal
    ]
      \addplot[only marks, mark=x, blue, line width=1.5pt, mark size=4.5pt] table {figure_data/parametric_delay_35.dat};
      \addplot[only marks, mark=o, red, line width=1.5pt, mark size=4.5pt] table {figure_data/parametric_delay_reference_35.dat};
      \addplot[only marks, mark=diamond, green!60!black, line width=1pt, mark size=4pt] table {figure_data/parametric_delay_left_interpolation.dat};
      \addplot[only marks, mark=diamond, purple!60!black, line width=1pt, mark size=4pt] table {figure_data/parametric_delay_right_interpolation.dat};
      \draw[line width=1pt, black, dashed] (axis cs:0,0) circle [radius=0.075];
    \end{axis}
  \end{tikzpicture}
\end{subfigure}%
\qquad
\begin{subfigure}[t]{0.45\textwidth}
  \centering
  \begin{tikzpicture}
    \begin{axis}[
      title={$p = 50$},
      xlabel={$\Re(z)$},
      ylabel={$\Im(z)$},
      xtick={-0.1,0,0.1},
      ytick={-0.1,0,0.1},
      every axis y label/.style={at={(-0.18,.5)},rotate=90,anchor=center},
      xmin=-0.12, xmax=0.12,
      ymin=-0.12, ymax=0.12,
      line width=.75pt,
      width=2.8in,
      axis equal
    ]
      \addplot[only marks, mark=x, blue, line width=1.5pt, mark size=4.5pt] table {figure_data/parametric_delay_50.dat};
      \addplot[only marks, mark=o, red, line width=1.5pt, mark size=4.5pt] table {figure_data/parametric_delay_reference_50.dat};
      \addplot[only marks, mark=diamond, green!60!black, line width=1pt, mark size=4pt] table {figure_data/parametric_delay_left_interpolation.dat};
      \addplot[only marks, mark=diamond, purple!60!black, line width=1pt, mark size=4pt] table {figure_data/parametric_delay_right_interpolation.dat};
      \draw[line width=1pt, black, dashed] (axis cs:0,0) circle [radius=0.075];
    \end{axis}
  \end{tikzpicture}
\end{subfigure}

\begin{center}
\begin{tikzpicture}
  \begin{axis}[
    hide axis,
    xmin=0, xmax=1, ymin=0, ymax=1,
    legend columns=5,
    legend style={at={(0.5,1)}, anchor=south, font=\normalsize, /tikz/every even column/.append style={column sep=1em}},
  ]
    \addlegendimage{only marks, mark=x, blue, line width=1.5pt, mark size=4.5pt}
    \addlegendimage{only marks, mark=o, red, line width=1.5pt, mark size=4.5pt}
    \addlegendimage{only marks, mark=diamond, green!60!black, line width=1pt, mark size=4pt}
    \addlegendimage{only marks, mark=diamond, purple!60!black, line width=1pt, mark size=4pt}
    \addlegendimage{line width=1pt, black, dashed}
    \legend{Computed $\lambda_j$, True $\lambda_j$, $\theta_i$, $\sigma_i$, Contour $\partial\Omega$}
  \end{axis}
\end{tikzpicture}
\end{center}
\caption{Eigenvalues for the delay pNLEVP example. The true eigenvalues are known for this problem setup and allow for a comparison. While only samples in $\cP = [30,35]$ were used in the offline phase of the parametric multipoint Loewner method, eigenvalues for $p=20$ and $p=50$ are also approximated accurately. During the online phase with $p\not\in\pdom$, computed eigenvalues need not fall in the domain $\Omega$ used in the offline phase: in the top-left plot, two of the predicted eigenvalues fall outside $\Omega$, while in the bottom-right plot, $\Omega$ contains two extra eigenvalues that are not approximated.}
\label{figure:delayEigenvalues}
\end{figure}

We proceed by computing eigenvalue and eigenvector approximations for the parameter values $p = 30,35\in\cP$ and $p=20,50\not\in\cP$. Figure~\ref{figure:delayEigenvalues} depicts the approximated eigenvalues and contour $\partial \Omega$ for these parameter values. Note that for $p \in \cP$ we have exactly $4$ eigenvalues inside $\Omega$, reflected in the plots for $p=30,35$. For $p=20$, two of the computed eigenvalues fall outside the domain $\Omega$ but are still captured accurately. We emphasize that the requirement that the number of eigenvalues in $\Omega$ remains constant only applies to the offline phase of Algorithm~\ref{alg:pmpl} with $p\in\cP$.\ \ During the online phase, especially with $p\not\in\pdom$, this is not explicitly required. However, as we showed in the previous example, the approximation quality may deteriorate quickly for parameter values outside $\cP$, especially near branch points. A similar discussion applies for $p=50$. In this case, two additional eigenvalues move inside the domain $\Omega$. However, since we only captured four eigenvalues during the offline phase of Algorithm~\ref{alg:pmpl}, we will only be able to compute approximations for four eigenvalues during the online phase for this example, regardless of how $p$ is chosen.

\subsection{Damped string example}
\label{sec:dampedString}
We next consider a model for a string vibrating over the physical domain $[0,1]$ with homogeneous Dirichlet boundary conditions and constant viscous damping of strength $p\ge 0$ applied on only the middle half of the domain, $[0.25,0.75]$~\cite{Emb25}.  Writing $\hz=\sqrt{z^2 + 2pz}$ for convenience, we have
\begin{equation*}
    \bT(z,p) = \begin{bmatrix}
        -\sinh(z/4) & \sinh(\hz/4) & \cosh(\hz/4) & 0 \\
        -z \cosh(z/4) & \hz \cosh(\hz/4) & \hz \sinh(\hz/4) & 0 \\
        0 & -\sinh(3\hz/4) & - \cosh(3\hz/4) & \sinh(z/4) \\
        0 & -\hz \cosh(3\hz/4) & -\hz \sinh(3\hz/4) & -z\cosh(z/4)
    \end{bmatrix}.
\end{equation*}
For $p=0$ (no damping), all the eigenvalues are simple, appearing as complex conjugate pairs on the imaginary axis. As $p$ increases, the eigenvalues move into the complex left-half plane, tending toward the real axis. We seek the parameter that maximizes the asymptotic rate of decay of the energy in the string, given by the value of $p\ge 0$ that minimizes the spectral abscissa (real part of the rightmost eigenvalue)~\cite{CZ94}. For each $p \in \cP$, $\bT(\cdot,p)$ has infinitely many eigenvalues. For our investigation it will suffice to only consider a finite number of eigenvalues that are close to the real axis.

An important difficulty arises for this problem: $\bT(\cdot,p)$ is not analytic in $\C$ due to branch points at $z=0$ and $z = -2p$, i.e., the zeros of the expression inside the root $\hz = \sqrt{z^2 + 2pz}$. Hence, we can specify a branch cut via a curve that connects $z=0$ and $z=-2p$, which then allows us to define two branches of the multivalued function $\bT(\cdot,p)$ associated with that branch cut. The two simplest choices are the line connecting $0$ and $-2p$ along the real axis through $-p$, and the line connecting $0$ and $-2p$ along the real axis through $\infty$. The first option seems simpler, but would interfere with the desired eigenvalues; we thus use the second option for our computations. (Note that the branch cut moves as we vary $p$,  an additional wrinkle of this problem.) 

We consider two cases, both taking the domain $\Omega$ to be the interior of a complex ellipse with center $c \in \C$, real semi-axis $s_r$ and imaginary semi-axis $s_i$:
\begin{equation*}
    \cE(c,s_r,s_i) := \left\{ z \in \mathbb{C}: 
\left( \frac{\operatorname{Re}(z - c)}{s_r} \right)^2 
+ \left( \frac{\operatorname{Im}(z - c)}{s_i} \right)^2 
< 1 \right\}.
\end{equation*}
Denote the boundary as $\partial \cE$. For the first case, take $\cP = [3,4]$ and $\Omega = \cE(-3,2.5,10)$; for the second case, take $\cP = [4,5]$ and $\Omega = \cE(-2,1.75,15)$. For the first case, we use $500$ uniformly sampled points along $\partial \cE(-3,3,11)$ for the set $\{ s_i \}_{i=1}^{500}$ and $25$ uniformly spaced points $\{ p_j \}_{j=1}^{25}$ in $\cP$.\ \ For the second case, we use $500$ uniformly sampled points along $\partial \cE(-2,2,16)$ for the set $\{ s_i \}_{i=1}^{500}$ and $25$ uniformly spaced points $\{ p_j \}_{j=1}^{25}$ in $\cP$.\ \ As for the last two examples, we use interlaced points from $\{ s_i \}_{i=1}^{500}$ as $\{ \theta_i \}_{i=1}^{250}$ and $\{ \sigma_i \}_{i=1}^{250}$. Figure~\ref{figure:dampedStringFixedDomain} shows $\Omega$ and the eigenvalue approximations for both cases.

The plot for the first case shows that we accurately capture a pair of eigenvalues that exhibits nontrivial behavior: it starts off as a complex conjugate pair for $p=3$, turns into a real eigenvalue with algebraic multiplicity two for $p \approx 3.71$, and is finally given by two real eigenvalues at $p=4$. For the second case, none of the computed eigenvalue curves coalesce: the eigenvalues remain simple in $\Omega$.\ \ For this case we seek the value of $p$ that minimizes the spectral abscissa, apparently near $p \approx 4.71$. To verify the accuracy of our computed results, we compute eigenvalue residuals for various parameter values. In particular, for both cases we consider $200$ uniformly spaced test parameters in $\cP$ for which we compute the maximum of the eigenvalue residuals $\lVert \bT(\lambda(p),p) \bv(p) \rVert_2 / \lVert \bv(p) \rVert_2$. For the first case the residual remains below $3 \times 10^{-10}$ and for the second case it stays below $3 \times 10^{-8}$.

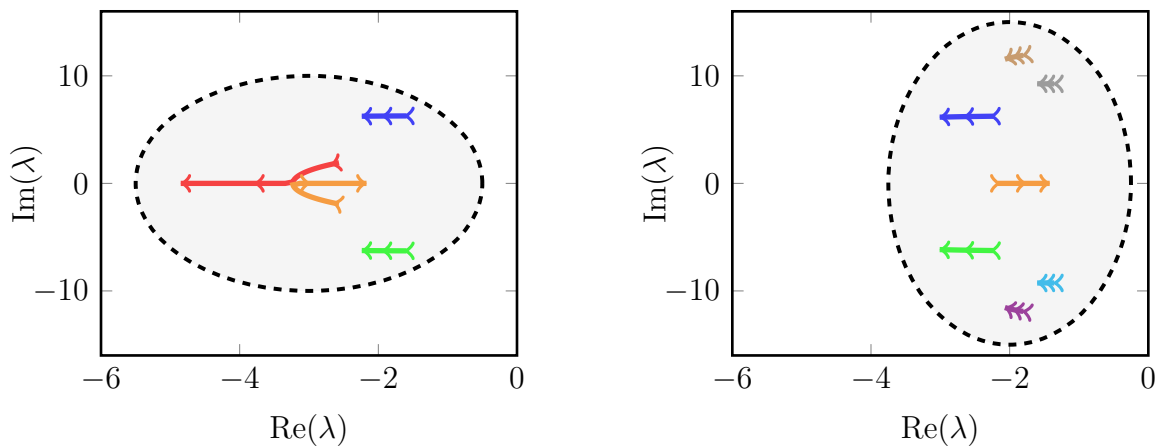
\begin{figure}
    \centering
    \begin{subfigure}{0.45\textwidth}
    \begin{tikzpicture}
      \begin{axis}[
        axis line style={line width=1pt},
        legend style={line width=1pt},
        width=2.8in,
        xmin=-6, xmax=0,
        ymin=-16, ymax=16,
        xlabel={$\Re(\lambda)$},
        ylabel={$\Im(\lambda)$},
        every axis y label/.style={at={(-0.18,.5)},rotate=90,anchor=center},
        line width=1.5pt,
        every axis plot/.style={line width=2pt},
        legend columns=2,
        legend pos=north west,
        legend cell align=left,
      ]
    
        \tikzset{
          arrow on path/.style={
            postaction={
              decorate,
              decoration={
                markings,
                  mark=at position 0.05 with {\arrow[line width=1.5pt]{>}},
                  mark=at position 0.55 with {\arrow[line width=1.5pt]{>}},
                  mark=at position 1 with {\arrow[line width=1.5pt]{>}}
              }
            }
          }
        }

        \addplot[blue, arrow on path] table [x index=0, y index=4] {figure_data/parametric_wave_coalesce.dat};
        \addplot[green, arrow on path] table [x index=1, y index=5] {figure_data/parametric_wave_coalesce.dat};
        \addplot[red, arrow on path] table [x index=2, y index=6] {figure_data/parametric_wave_coalesce.dat};
        \addplot[orange, arrow on path] table [x index=3, y index=7] {figure_data/parametric_wave_coalesce.dat};\
        \filldraw[fill=gray!25, fill opacity=0.3, draw=none, dashed]
  (axis cs:-3,0) ellipse [x radius=2.5, y radius=10];
        \draw[line width=1.5pt, black, dashed] (axis cs:-3,0) ellipse [x radius=2.5, y radius=10];
    
      \end{axis}
    \end{tikzpicture}
    \end{subfigure}
    \qquad
    \begin{subfigure}{0.45\textwidth}
    \begin{tikzpicture}
      \begin{axis}[
        axis line style={line width=1pt},
        legend style={line width=1pt},
        width=2.8in,
        xmin=-6, xmax=0,
        ymin=-16, ymax=16,
        xlabel={$\Re(\lambda)$},
        ylabel={$\Im(\lambda)$},
        every axis y label/.style={at={(-0.18,.5)},rotate=90,anchor=center},
        line width=1.5pt,
        every axis plot/.style={line width=2pt},
        legend columns=2,
        legend pos=north west,
        legend cell align=left,
        ]

        \tikzset{
          arrow on path/.style={
            postaction={
              decorate,
              decoration={
                markings,
                  mark=at position 0.05 with {\arrow[line width=1.5pt]{>}},
                  mark=at position 0.55 with {\arrow[line width=1.5pt]{>}},
                  mark=at position 1 with {\arrow[line width=1.5pt]{>}}
              }
            }
          }
        }

        \addplot[brown, arrow on path] table [x index=0, y index=7] {figure_data/parametric_wave_optimal_damping.dat};
        \addplot[gray, arrow on path] table [x index=1, y index=8] {figure_data/parametric_wave_optimal_damping.dat};
        \addplot[blue, arrow on path] table [x index=2, y index=9] {figure_data/parametric_wave_optimal_damping.dat};
        \addplot[orange, arrow on path] table [x index=3, y index=10] {figure_data/parametric_wave_optimal_damping.dat};
        \addplot[violet, arrow on path] table [x index=4, y index=11] {figure_data/parametric_wave_optimal_damping.dat};
        \addplot[green, arrow on path] table [x index=5, y index=12] {figure_data/parametric_wave_optimal_damping.dat};
        \addplot[cyan, arrow on path] table [x index=6, y index=13] {figure_data/parametric_wave_optimal_damping.dat};
        \filldraw[fill=gray!25, fill opacity=0.3, draw=none, dashed]
  (axis cs:-2,0) ellipse [x radius=1.75, y radius=15];
        \draw[line width=1.5pt, black, dashed] (axis cs:-2,0) ellipse [x radius=1.75, y radius=15];
      \end{axis}
    \end{tikzpicture}
    \end{subfigure}

\vspace*{-7pt}
    \caption{Eigenvalues and domains for the two cases of the damped string example. In the first case (left plot) for $\cP = [3,4]$, a defective eigenvalue appears at $p \approx 3.71$. The parametric multipoint Loewner method is able to handle this case naturally, without further input. In the second case (right plot) for $\cP = [4,5]$, the spectral abscissa appears to be minimized for $p \approx 4.71$.}
    \label{figure:dampedStringFixedDomain}
\end{figure}

\section{Conclusions}
\label{sec:conclusions}
Parameterized nonlinear eigenvalue problems pose a computational challenge.  As a platform for algorithm development, we propose a parametric extension to Keldysh's theorem, which decomposes a pNLEVP function into the form $\bT(z,p) = \bR(z,p)/u(z,p) + \bN(z,p)$.  The eigenvalue of interest are captured by the functions $\bR(z,p)$ and $u(z,p)$, which are analytic in the domain of interest, $\Omega\times \pdom$.\ \ In contrast to a pointwise application of the standard Keldysh decomposition for each $p\in\pdom$, our extension naturally captures the $p$-dependence of the eigenvalues, allowing for scenarios where the Jordan structure changes with $p$.  Contour integral methods based on rational approximation build nicely upon this foundation.  Using the multipoint Loewner framework, we propose an algorithm for approximating the eigenvalues in $\Omega \times \pdom$.\ \ When $\bR(z,p)$ and $u(z,p)$ are rational functions, this method can in principle recover the eigenvalues exactly (assuming the contour integration and other computations are exact);  when those functions are not rational, they can be approximated to high accuracy with rational approximants.  Our experiments show that such approximations can still yield excellent approximations of the eigenvalue.  

Opportunities for further research abound. One could extend our main result to allow for domains $\Omega(p)$ that evolve with the parameter $p$.  Algorithmically, one might consider optimal methods for choosing the parameter sample points $p_j \in \pdom$, or relaxing the need for samples at all combinations of interpolation points $s_i$, parameter samples $p_j$, and directions $\bell_k$ and $\br_k$.  One might also consider variations that extend contour integrals algorithms based on Markov parameters, as described in~\cite{asakura2009,beyn2012,brennan2023}, as well as the potential for convergence beyond the target parameter domain $\Pi$.

\appendix
\section{Keldysh's Theorem}
\label{sec:generalKeldysh}
Here we state Keldysh's Theorem as given by G\"uttel and Tisseur in \cite[thm.~2.8]{guttel2017}, including the eigenvector normalizations. In the same survey the authors precisely define algebraic, geometric and partial multiplicity, along with generalized eigenvectors and their connection to Jordan chains and invariant pairs. We omit these definitions here, focusing on the aspects most relevant to our contributions. For detailed proofs of the statements in the theorem, see~\cite{mennicken2003}, where $\bV (z \bI - \bJ)^{-1} \bW^*$ is given in an equivalent expanded form.
\begin{theorem}[Keldysh, as presented in~\cite{guttel2017}]
    \label{theorem:generalKeldysh}
    Let $\bT : \Omega \rightarrow \C^{n \times n}$ be analytic in the nonempty domain $\Omega \subset \C$, and assume $\det(\bT(\cdot)) \not\equiv 0$. Let $\lambda_1,\ldots,\lambda_s$ be the distinct eigenvalues of $\bT$ in $\Omega$. For $i = 1,\ldots, s$ let $m_{i1} \geq \cdots \geq m_{id_i}$ denote the partial multiplicities of $\lambda_i$ with geometric multiplicity $d_i$, and denote the total number of eigenvalues (counting multiplicities) as
    \begin{equation*}
        m = \sum_{i=1}^s \sum_{j=1}^{d_i} m_{ij}.
    \end{equation*}
    Then there exist matrices
    \begin{alignat*}{3}
        \bV &= \left[ \bV_1, \ldots, \bV_s \right] \in \C^{n \times m},  &\  \bV_i &= \left[ \bV_1^{(i)}, \ldots, \bV_{d_i}^{(i)} \right],  &\ \bV_j^{(i)} &= \left[ \bv_1^{(i,j)}, \ldots, \bv_{m_{ij}}^{(i,j)} \right], \\
        \bW &= \left[ \bW_1, \ldots, \bW_s \right] \in \C^{n \times m},  &\ \bW_i &= \left[ \bW_1^{(i)}, \ldots, \bW_{d_i}^{(i)} \right],  &\ \bW_j^{(i)} &= \left[ \bw_{m_{ij}}^{(i,j)}, \ldots, \bw_1^{(i,j)} \right],
    \end{alignat*}
    whose columns are generalized right and left eigenvectors $\bv_k^{(i,j)}$ and $\bw_k^{(i,j)}$, and
    \begin{equation*}
        \bJ = \operatorname{diag}(\bJ_1,\ldots,\bJ_s) \in \C^{m \times m}, \qquad \bJ_i = \operatorname{diag}(\bJ_i^{(1)},\ldots,\bJ_i^{(d_i)}),
    \end{equation*}
    where
    \begin{equation*}
        \bJ_i^{(j)} = \begin{bmatrix}
            \ \lambda_i & 1 & & \\[-1pt]
            & \lambda_i & \ddots & \\[-1pt]
            & & \ddots & 1\ \\[-1pt]
            & & & \lambda_i\  
        \end{bmatrix} \in \C^{m_{ij} \times m_{ij}},
    \end{equation*}
    such that
    \begin{equation*}
        \bT(z)^{-1} = \bV (z \bI - \bJ)^{-1} \bW^* + \bN(z)
    \end{equation*}
    for some $\bN: \Omega \rightarrow \C^{n \times n}$ analytic in $\Omega$.\ \ The eigenvectors are normalized such that
    \begin{equation}
    \label{eq:eigenvectorNormalization}
        \sum_{\alpha=0}^t \sum_{\beta=1}^{m_{ij}} \left( \bw_{t-\alpha + 1}^{(i,k)} \right)^* \frac{\bT^{(\alpha+\beta)}(\lambda_i)}{(\alpha + \beta)!} \bv_{m_{ij}-\beta+1}^{(i,j)} = \delta_{jk} \delta_{0t}
    \end{equation}
    for $t=0,\ldots,m_{ik}-1$, and $j,k = 1,\ldots,d_i$, and $i = 1,\ldots,s$, where 
    \begin{equation*}
        \delta_{jk} = \begin{cases}
        1, & \text{if } j = k; \\
        0, & \text{if } j \ne k.
        \end{cases}
    \end{equation*}
\end{theorem}

\section{Rational approximation in line 2 of Algorithm~\ref{alg:pmpl}}
\label{sec:rationalApproximation}
Here we outline our proposed method for rational approximation of the functions $\bell_k^\top \bH(z,p)$ and $\bH(z,p) \br_k$ that constitute Line~2 of Algorithm~\ref{alg:pmpl}. We primarily rely on the p-AAA approach for approximating vector-valued data from~\cite{rodriguez2023}. For scalar-valued data, p-AAA uses a set of samples $\{ \cD(s_i,p_j) \}$ and sampling values $\{ s_i \}_{i=1}^{2r}, \{ p_j \}_{j=1}^q \subset \C$ to compute a rational function $f$ in the barycentric form \eqref{eq:barycentricForm} such that $f(z,p) \approx \cD(z,p)$. The algorithm achieves this goal by interpolating the data at select nodes $\{\xi_i\} \subset \{ s_i \}_{i=1}^{2r}$ and $\{\pi_j\} \subset \{ p_j \}_{j=1}^q$, and then solving a linear least-squares problem to find coefficients $\alpha_{ij}$ such that the non-interpolated data is approximated in a least-squares sense. See~\cite{rodriguez2023} for details. An important property of p-AAA is that, under certain conditions, the algorithm will exactly recover a rational function from data, in which case it is equivalent to the parametric Loewner framework \cite{ionita2014}.

Here, we focus on the p-AAA formulation from~\cite{rodriguez2023} for matrix-valued multivariate rational approximation. Suppose we would like to approximate a matrix-valued function $\bH$ using samples in $\{ \bH(s_i, p_j) \}$. The first step of the idea in \cite{rodriguez2023} is to generate random vectors $\bell$ and $\br$, then approximate the samples in $\{ \cD(s_i,p_j) = \bell^\top \bH(s_i, p_j) \br \}$ via a scalar-valued multivariate rational function, as in \eqref{eq:barycentricForm}. The interpolated values $\cD(\xi_i,\pi_j)$ in \eqref{eq:barycentricForm} are then replaced by the corresponding matrix values $\bH(\xi_i,\pi_j)$, resulting in a matrix-valued interpolant. Under certain conditions, the matrix-valued rational function $\bH$ can be recovered exactly with this approach.

In Algorithm~\ref{alg:rationalApproximation}, we propose a method for rational approximation of $\bell_k^\top \bH(z,p)$ and $\bH(z,p)@\br_k$ that follows a similar approach but replaces $\cD(\xi_i,\pi_j)$ with the probed data $\bell_k \bH(\xi_i, \pi_j)$ and $\bH(\xi_i, \pi_j)@ \br_k$. Since we only anticipate access to probed data to begin with, in \eqref{eq:scalarApproximationData} we propose constructing the scalar data $\cD(s_i,p_j)$ using $\bell$ and $\br$ that are linear combinations of the given probing vectors. We suggest using the mean $\bell = (\bell_1 + \cdots + \bell_r)/r$ and $\br = (\br_1 + \cdots + \br_r)/r$, as this choice performed well in practice (although other linear combinations can be chosen). We emphasize that this construction does not require any matrix-valued samples of the form $\bH(s_i,p_j)$. Recall our important assumption that, for each $p \in \cP$, the number of eigenvalues in $\Omega$ does not change. We verify this assumption by computing the rank of several Loewner matrices~\eqref{eq:eigenvalueAssumptionCheck}, checking that these ranks coincide (giving the number of eigenvalues, $m$, in $\Omega$). This information can then also be incorporated into the p-AAA algorithm to ensure that the computed function $f$ in \eqref{eq:barycentricForm} has the correct number of poles in $z$.

\begin{algorithm}[t!]
    \caption{Rational Approximation for Algorithm~\ref{alg:pmpl}}
    \label{alg:rationalApproximation}
    \begin{algorithmic}[1]
        \Require{Sampling values $\{s_i\}_{i=1}^{2r} \subset \C$, $\{p_j\}_{j=1}^q \subset \cP$, probing vectors $\{ \bell_k \}_{k=1}^r$, $\{ \br_k \}_{k=1}^r\subset \C^n$, approximate probed samples $ \bell_k^\top \bH(s_i, p_j) $, $\bH(s_i, p_j)@\br_k$ from step~1 of Algorithm~\ref{alg:pmpl}}
        \Ensure{$\bL_1,\ldots,\bL_r:\C\times \C \to \C^{1\times n}$ and $\bR_1,\ldots,\bR_r:\C\times \C \to \C^{n\times 1}$}
        \State Partition $\{s_i\}_{i=1}^{2r} = \{\theta_i\}_{i=1}^r \cup \{\sigma_i\}_{i=1}^r$ and for each $p \in \{p_i\}_{i=1}^q$ determine the number of eigenvalues $m$ in $\Omega$ as the common rank of the Loewner matrices
        \begin{equation}
        \label{eq:eigenvalueAssumptionCheck}
            \L_{ij}(p) = \frac{\bell_i^\top [\bH(\theta_i,p) - \bH(\sigma_j,p)] \br_j}{\theta_i - \sigma_j}.
        \end{equation}
        \State For $\bell, \br$ chosen as a linear combination of $\bell_1,\ldots,\bell_r$ and $\br_1,\ldots,\br_r$, use the probed samples $\bell_k^\top \bH(s_i, p_j)$ and $\bH(s_i, p_j)@\br_k$ to construct the data
        \begin{equation}
            \label{eq:scalarApproximationData}
            \cD(s_i,p_j) = \bell^\top \bH(s_i,p_j)@\br.
        \end{equation}
        \State Use the p-AAA algorithm from \cite{rodriguez2023} to compute the nodes $\{\xi_i\} \in \{ s_i \}_{i=1}^{2r}$, $\{\pi_j\} \in \{ p_j \}_{j=1}^q$ and coefficients $\alpha_{ij} \in \C$ of the rational function
        \begin{equation}
        \label{eq:barycentricForm}
            f(z,p) = \sum_{i=1}^{m+1} \sum_{j=1}^{c+1} \frac{\alpha_{i j} \cD(\xi_i,\pi_j)}{(z - \xi_i)(p - \pi_{j})} \Bigg/  \sum_{i=1}^{m+1} \sum_{j=1}^{c+1} \frac{\alpha_{i j}}{(z - \xi_i)(p - \pi_{j})}
        \end{equation}
        that approximates the samples $\cD(s_i,p_j)$. 
        \State Using the computed nodes and coefficients, form the vector-valued functions
        \begin{equation}
        \label{eq:vectorValuedRationalApproximants}
        \begin{aligned}
            \bL_k(z,p) &\!=\! \sum_{i=1}^{m+1} \sum_{j=1}^{c+1} \frac{\alpha_{i j} \bell_k^\top \bH(\xi_i,\pi_j)}{(z \!-\! \xi_i)(p \!-\! \pi_{j})} \Bigg/  \sum_{i=1}^{m+1} \sum_{j=1}^{c+1} \frac{\alpha_{i j}}{(z \!-\! \xi_i)(p \!-\! \pi_{j})} \approx \bell_k^\top \bH(z,p)\\
            \bR_k(z,p) &\!=\! \sum_{i=1}^{m+1} \sum_{j=1}^{c+1} \frac{\alpha_{i j} \bH(\xi_i,\pi_j)@\br_k}{(z \!-\! \xi_i)(p \!-\! \pi_{j})} \Bigg/  \sum_{i=1}^{m+1} \sum_{j=1}^{c+1} \frac{\alpha_{i j}}{(z \!-\! \xi_i)(p \!-\! \pi_{j})} \approx \bH(z,p) @\br_k.
        \end{aligned}
        \end{equation}
    \end{algorithmic}
\end{algorithm}

\section{Proof of the parametric Keldysh decomposition}
\label{sec:proofAppendix}
In this section, we present the proof of Theorem~\ref{theorem:parametricKeldysh}.

\subsection{Preliminaries for the two-variable case}
We first review some standard results and establish key preliminaries needed for the proof. 
 Note that a single-variable analytic function $f$ with a zero of multiplicity $m$ at $z_0\in\C$ can be written as
\begin{equation} \label{eq:svWPT}
    f(z) = (z-z_0)^m h(z),
\end{equation}
for some nonzero function $h$ analytic near $z_0$.  This simple fact becomes more involved in the two-variable case; it is described by the Weierstrass preparation theorem (see e.g., \cite{ebeling2007,gunning2015} for formulations). In the upcoming theorems and proofs, \emph{Weierstrass polynomials} will play an important role.
\begin{definition}[Weierstrass polynomial]
    \label{def:Weierstrass}
    For $j=0,\ldots,k-1$, let $\alpha_j: \C \rightarrow \C$ be analytic near $0$, and $\alpha_j(0) = 0$. We call a function $u: \C\times\C \rightarrow \C$ of the form
    \begin{equation}
        \label{eq:WeierstrassPolynomial}
        u(z,p) = z^k + \alpha_{k-1}(p) z^{k-1} + \cdots + \alpha_1(p) z + \alpha_0(p)
    \end{equation}
    a Weierstrass polynomial of degree $k$. Further, if $v: \C\times\C \rightarrow \C$ is a function such that $v(z+z_0,p+p_0)$ is a Weierstrass polynomial for some $(z_0,p_0) \in \C\times\C$, we call $v$ a shifted Weierstrass polynomial, and say that $v$ is a Weierstrass polynomial at $(z_0,p_0)$.
\end{definition}
Weierstrass polynomials are polynomials in the $z$ variable, and hence $u(\cdot,p)$ is analytic in $\C$ for any choice of $p$. They are also analytic in a neighborhood of $0$ in the $p$ variable. One can regard a Weierstrass polynomial as a polynomial in $z$ with coefficients varying analytically in $p$.  The zeros of a polynomial vary continuously with its coefficients (see, e.g., \cite{zedek_continuity_1965}), a fact that can be applied to prove the following proposition.
\begin{proposition}
    \label{proposition:ContZeros}
    Suppose $u$ is a Weierstrass polynomial of degree $k$, analytic in the neighborhood $\Omega \times \pdom$ of $(0,0)$. For any given neighborhood $\Omega' \subset \Omega$ of $0$ there exists a neighborhood $\pdom'$ of $0$ such that $u$ is analytic in $\Omega' \times \pdom'$ and, for all $p \in \pdom'$, the $k$ zeros of $u(\cdot,p)$ lie in $\Omega'$.
\end{proposition}
In general, the zeros of $u(\cdot,p)$ may lie anywhere in $\C$, since $u(\cdot,p)$ is a polynomial. However, since the zeros of a polynomial depend continuously on its coefficients, we can choose the parameter domain $\pdom'$ sufficiently small that all $k$ zeros of $u(\cdot,p)$ remain arbitrarily close to $0$ for all $p \in \pdom'$. This argument can be applied to Weierstrass polynomials, since $u(z,0) = z^k$, and hence for $p$ close to $0$, the $k$ zeros of $u(\cdot,p)$ move within a neighborhood of $0$. Proposition~\ref{proposition:ContZeros} follows from this observation. In the following, we will specify the domain of analyticity of (shifted) Weierstrass polynomials in several contexts. Based on the discussion above, we will assume (without loss of generality) that all zeros of $u(\cdot,p)$ are contained in the specified domain of analyticity. If necessary, the domain $\pdom'$ can always be shrunk to ensure that this property holds.

\bigskip
\begin{theorem}[Weierstrass preparation theorem]
    \label{theorem:Weierstrass}
    Suppose $f: \C\times\C \rightarrow \C$ is analytic in a neighborhood $\Omega \times \pdom$ of $(0,0)$, $f(0,0) = 0$ and $f(\cdot,0) \not\equiv 0$ in $\Omega$. There exists a neighborhood $\Omega' \times \pdom'$ of $(0,0)$, a unique Weierstrass polynomial $u$ analytic in $\Omega' \times \pdom'$, and a function $h$ analytic in $\Omega' \times \pdom'$ with $h(0,0) \neq 0$ such that for all $(z,p) \in \Omega'\times \pdom'$,
    \begin{equation} \label{eq:WPT}
        f(z,p) = u(z,p)@h(z,p).
    \end{equation}
\end{theorem}
The Weierstrass preparation theorem extends the single-variable decomposition~\eqref{eq:svWPT} to multivariate analytic functions, replacing $z^{k}$ with the Weierstrass polynomial $u(z,p)$ introduced in \eqref{eq:WeierstrassPolynomial}. The following lemma clarifies this connection, showing that the order of the zero of $f(\cdot,0)$ must match the polynomial degree of $u(\cdot,0)$.
\begin{lemma}
    \label{lemma:WeierstrassZeroOrder}
    In the context of Theorem~\ref{theorem:Weierstrass}, if $f(\cdot,0)$ has a zero of order $k$ at $0$, then the polynomial degree of the unique Weierstrass polynomial $u$ in \eqref{eq:WPT} is $k$.
\end{lemma}
\begin{proof}
    First, note that $f(\cdot,0)$ is analytic in the neighborhood $\Omega'$ of $0$. Since $f(\cdot,0)$ has a zero of order $k$ at $0$, there exists a function $\widehat{h} : \C \rightarrow \C$ analytic and nonzero in $\Omega'$ such that $f(z,0) = z^k \widehat{h}(z)$. Now apply the Weierstrass preparation theorem to obtain $f(z,p) = u(z,p) h(z,p)$ as in Theorem~\ref{theorem:Weierstrass}, where $u$ is Weierstrass polynomial of degree $\widehat{k}$. Note that $u(z,0) = z^{\widehat{k}}$ (see Definition~\ref{def:Weierstrass}). In particular, for $z$ near $0$,
    \begin{equation*}
        f(z,0) = z^k \widehat{h}(z) = z^{\widehat{k}} h(z,0),
    \end{equation*}
    which can be rearranged to give
    \begin{equation*}
        h(z,0) = z^{k - \widehat{k}}@@\widehat{h}(z).
    \end{equation*}
We must have $k - \widehat{k} \le 0$; otherwise we would have $h(0,0) = 0$ (which is not true for $h$ originating from the Weierstrass preparation theorem). Similarly, we must have $k - \widehat{k} \geq 0$; otherwise $h(\cdot,0)$ would have a singularity at $0$ (which is not true, since Theorem~\ref{theorem:Weierstrass} states that $h$ is analytic near $(0,0)$). Thus, we must have $\widehat{k} = k$.
\end{proof}
Theorem~\ref{theorem:Weierstrass} can be recast for a zero located at an arbitrary $(z_0,p_0) \in \C\times\C$.
\begin{corollary}
    \label{corollary:shiftedWeierstrass}
    Suppose $f: \C\times \C\rightarrow \C$ is analytic in a neighborhood $\Omega \times \pdom$ of $(z_0,p_0)$, $f(z_0,p_0) = 0$ and $f(\cdot,p_0) \not\equiv 0$ in $\Omega$. There exists a neighborhood $\Omega' \times \pdom'$ of $(z_0,p_0)$, a unique Weierstrass polynomial $u$ at $(z_0,p_0)$ analytic in $\Omega' \times \pdom'$, and a function $h$ analytic in $\Omega' \times \pdom'$ with $h(z_0,p_0) \neq 0$ such that for all $(z,p) \in \Omega'\times \pdom'$,
    \begin{equation*}
        f(z,p) = u(z,p)@h(z,p).
    \end{equation*}
\end{corollary}
The Weierstrass preparation theorem is a local result about $(z_0,p_0)$.  We extend it to the case where $f$ could have a finite number of distinct zeros in $\Omega\times \pdom$.
\begin{lemma}
    \label{lemma:generalizedWeierstrass}
    Suppose $f: \C\times \C \rightarrow \C$ is analytic in a domain $\Omega \times \pdom$, and $f(\cdot,p)$ has exactly $m$ zeros (summing orders) for all $p \in \pdom$. There exists a unique polynomial
    \begin{equation*}
        u(z,p) = z^m + \alpha_{m-1}(p) z^{m-1} + \cdots + \alpha_{1}(p) z + \alpha_{0}(p)
    \end{equation*}
    with coefficients $\alpha_j : \C \rightarrow \C$ analytic in $\pdom$ and a function $h$ analytic and nonzero in $\Omega \times \pdom$ such that for all $(z,p) \in \Omega \times \pdom$,
    \begin{equation*}
        f(z,p) = u(z,p)@h(z,p).
    \end{equation*}
\end{lemma}
\begin{proof}
    Fix some $\pi \in \pdom$. Denote the distinct zeros of $f(\cdot,\pi)$ in $\Omega$ as $z_1,\ldots,z_k$ with corresponding orders $m_1,\ldots,m_k$ such that $m_1 + \cdots + m_k = m$. Corollary~\ref{corollary:shiftedWeierstrass} states that for each $j=1,\ldots,k$ there exists a neighborhood $\Omega_j \times \pdom_j \subset \Omega \times \pdom$ of $(z_j,\pi)$, a function $u_j$ that is a Weierstrass polynomial of degree $m_j$ at $(z_j,\pi)$ and analytic in $\Omega_j \times \pdom_j$, as well as a function $h_j$ analytic in $\Omega_j \times \pdom_j$ satisfying $h_j(z_j,\pi) \neq 0$ such that
    \begin{equation*}
        f(z,p) = u_j(z,p)@h_j(z,p) \quad \text{for} \quad (z,p) \in \Omega_j \times \pdom_j.
    \end{equation*}
    For the functions $u_j$ and $h_j$, we can make two assumptions without loss of generality. First, we assume that for all $p \in \pdom_j$, all $m_j$ zeros of $u_j(\cdot,p)$ lie in $\Omega_j$. (This assumption is justified by Proposition~\ref{proposition:ContZeros} and the discussion immediately following it.) Further, we assume that $h_j$ is not only nonzero at $(z_j,\pi)$, but also nonzero on the entire set $\Omega_j \times \pdom_j$. (Since $h_j$ is analytic, it must be nonzero in some neighborhood of $(z_j,\pi)$.) To ensure that these two properties, together with all previously imposed properties of $u_j$ and $h_j$, hold simultaneously, we can shrink the neighborhood $\Omega_j \times \pdom_j$ of $(z_j,\pi)$ if necessary by taking intersections of finitely many admissible neighborhoods. Based on these assumptions, the zero sets of $f$ and $u_j$ on $\Omega_j \times \pdom_j$ coincide, meaning
    \begin{equation*}
        \left\{ (z,p) \in \Omega_j \times \pdom_j: f(z,p) = 0 \right\} = \left\{ (z,p) \in \Omega_j \times \pdom_j: u_j(z,p) = 0 \right\}.
    \end{equation*}
    Consider the set $\pdom_{\pi} = \bigcap_{j=1}^k \pdom_j$. Since each $\pdom_j$ is a neighborhood of $\pi$, $\pdom_{\pi}$ is a nonempty neighborhood of $\pi$. Further, $\pdom_{\pi} \subseteq \pdom_j$ for $j=1,\ldots,k$ and hence all $u_j$ are analytic in $\Omega \times \pdom_{\pi}$, since all $u_j$ are polynomial in the first variable. Their product
    \begin{equation}
        \label{eq:uproduct}
        u(z,p) = u_1(z,p) \cdots u_k(z,p)
    \end{equation}
    is analytic in $\Omega \times \pdom_{\pi}$ and $u(\cdot,p)$ has exactly $m_1 + \cdots + m_k = m$ (i.e., the sum of the degrees of $u_1,\ldots,u_k$) zeros that must coincide with the $m$ zeros of $f(\cdot,p)$ in $\Omega$ for all $p \in \pdom_{\pi}$. Overall, we see that
    \begin{equation*}
        h(z,p) = \frac{f(z,p)}{u(z,p)}
    \end{equation*}
    is well-defined, analytic, and nonzero in $\Omega \times \pdom_{\pi}$. Note that $u(\cdot,p)$ is a monic polynomial, as it is the product of the monic polynomials $u_j(\cdot,p)$. So far, we have shown that, given $\pi \in \pdom$ there exists a neighborhood $\pdom_\pi$ of $\pi$ and functions $h_{\pi}$ and $u_{\pi}$ analytic in $\Omega \times \pdom_{\pi}$ with $h_{\pi}$ nonzero in $\Omega \times \pdom_{\pi}$, $u_{\pi}(\cdot,p)$ a degree-$m$ monic polynomial for all $p \in \pdom_{\pi}$, and $f = u_{\pi} h_{\pi}$ in $\Omega \times \pdom_{\pi}$. We complete the proof via an analytic continuation argument. Consider any $\pi,\pi' \in \pdom$ such that $\pdom_{\pi} \cap \pdom_{\pi'} = \pdom' \neq \emptyset$. The representations $f(z,p) = u_{\pi}(z,p) h_{\pi}(z,p)$ and $f(z,p) = u_{\pi'}(z,p) h_{\pi'}(z,p)$ coincide in $\Omega \times \pdom'$.\ \ Further, for all $p \in \pdom'$, the zeros of $u_{\pi}(\cdot,p)$, $u_{\pi'}(\cdot,p)$, and $f(\cdot,p)$ in $\Omega$, as well as their orders, coincide. Since $u_\pi(\cdot,p)$ and $u_{\pi'}(\cdot,p)$ are monic polynomials, they are identical. Since this holds for all $p \in \pdom'$, we have that $u_{\pi} = u_{\pi'}$ in $\Omega \times \pdom'$. By analytic continuation we thus have $u_{\pi} = u_{\pi'}$ in $\Omega \times \left(\pdom_{\pi} \cup \pdom_{\pi'}\right)$; we can also conclude that $h_{\pi} = h_{\pi'}$ is nonzero and analytic in $\Omega \times \left(\pdom_{\pi} \cup \pdom_{\pi'}\right)$. Now for an arbitrary $\pi'' \in \pdom$ we can choose a finite-length path that connects $\pi$ and $\pi''$ (since $\pdom$ is a domain, and thus connected) and repeat a similar argument with intersecting neighborhoods along that path (covering the path with overlapping domains, and taking a finite subcover) to show that $u_{\pi} = u_{\pi''}$ and $h_{\pi} = h_{\pi''}$. Since $\pi''$ is an arbitrary point in $\pdom$, we ultimately get that $f = u_{\pi} h_{\pi}$ throughout $\Omega \times \pdom$. This argument also shows uniqueness of $u_{\pi}$ and $h_{\pi}$, as analytic continuations are unique.
\end{proof}
We will also appeal to the Weierstrass division theorem~\cite{ebeling2007,gunning2015}.

\begin{theorem}[Weierstrass division theorem]
    Suppose $f: \C\times \C \rightarrow \C$ is analytic in a neighborhood of $(0,0)$, and $u$ is a Weierstrass polynomial of degree $m$. There exist unique functions $g$, $r$ analytic in some neighborhood of $(0,0)$, where $r(\cdot,p)$ is a polynomial of degree less than $m$, such that
    \begin{equation*}
        f(z,p) = g(z,p)@u(z,p) + r(z,p).
    \end{equation*}
\end{theorem}
The Weierstrass division theorem extends to a zero at an arbitrary $(z_0,p_0) \in \C\times \C$.
\begin{corollary}
    \label{corollary:shiftedWeierstrassDivision}
    Suppose $f: \C\times\C \rightarrow \C$ is analytic in a neighborhood of $(z_0,p_0)$, and $u$ is a Weierstrass polynomial of degree $m$ at $(z_0,p_0)$. There exist unique functions $g$, $r$ analytic in some neighborhood of $(z_0,p_0)$, where $r(\cdot,p)$ is a polynomial of degree less than $m$ such that
    \begin{equation*}
        f(z,p) = g(z,p)@u(z,p) + r(z,p).
    \end{equation*}
\end{corollary}
We next generalize the Weierstrass division theorem for several distinct zeros.
\begin{lemma}
    \label{lemma:f2division}
    Suppose $f, u : \C\times\C \rightarrow \mathbb{C}$ are analytic in the domain $\Omega \times \pdom$, and that for each $p\in\pdom$, $u(\cdot,p)$ is a polynomial of degree $m$. Further, suppose that for each fixed $\pi \in \Pi$ there exist points $z_1,\ldots,z_k \in \Omega$ and analytic functions 
    $u_1,\ldots,u_k : \C\times\C \rightarrow \mathbb{C}$ that are Weierstrass polynomials at $(z_1,\pi)$, $\ldots$, $(z_k,\pi)$, such that
    \begin{equation*}
        u(\cdot,p) = u_1(\cdot,p)\cdots u_k(\cdot,p)
    \end{equation*}
    for all $p$ in some neighborhood of $\pi$. Note that $k$ may depend on $\pi$. Then there exist unique analytic functions $g, r : \Omega \times \pdom \rightarrow \C$ such that for each $p \in \pdom$,
    the function $r(\cdot,p)$ is a polynomial of degree less than $m$, and
    \begin{equation*}
        f(z,p) = g(z,p)@u(z,p) + r(z,p).
    \end{equation*}
\end{lemma}
\begin{proof}
    Let $\pi \in \pdom$, and assume $u_1,\ldots,u_k$ are Weierstrass polynomials at $(z_1,\pi)$, $\ldots$, $(z_k,\pi)$. We then apply Corollary~\ref{corollary:shiftedWeierstrassDivision} to the function $f$ and $u_1$ to obtain
    \begin{equation}
        \label{eq:fdivision}
        f(z,p) = g_1(z,p)@u_1(z,p) + r_1(z,p),
    \end{equation}
    where $g_1$ and $r_1$ are analytic in a neighborhood $\Omega_1 \times \pdom_1$ of $(z_1,\pi)$. Note that $r_1(\cdot,p)$ is a polynomial, and hence $r_1$ is analytic in $\Omega \times \pdom_1$. By rearranging \eqref{eq:fdivision}, we see that the function $g_1$, analytic on $\Omega_1\times \pdom_1$, can be written as
    \begin{equation*}
        g_1(z,p) = \frac{f(z,p) - r_1(z,p)}{u_1(z,p)}.
    \end{equation*}
    Thus for all $p \in \pdom_1$ the zeros of the polynomial $u_1(\cdot,p)$ must cancel with the zeros of the analytic function $f(\cdot,p) - r_1(\cdot,p)$ in $\Omega_1$. Further, we assume (without loss of generality, based on Proposition~\ref{proposition:ContZeros}) that all zeros of $u_1(\cdot,p)$ lie in $\Omega_1$ for all $p \in \pdom_1$. Thus, $(f(z,p) - r_1(z,p)) / u_1(z,p)$ is not only analytic in $\Omega_1 \times \pdom_1$, but also on the larger domain $\Omega \times \pdom_1$. Hence, $g_1(z,p)$ is analytic in $\Omega \times \pdom_1$. We move to the next zero $(z_2,\pi)$ and apply Corollary~\ref{corollary:shiftedWeierstrassDivision} to $g_1$ and $u_2$ to obtain
    \begin{equation}
        \label{eq:gdivision}
        g_1(z,p) = g_2(z,p)@u_2(z,p) + r_2(z,p),
    \end{equation}
    where $g_1,g_2,r_2$ are all analytic on $\Omega \times \pdom_2$ for some neighborhood $\pdom_2$ of $\pi$.
    Repeat this process to obtain $g_{j-1}(z,p) = g_j(z,p) u_j(z,p) + r_j(z,p)$ for $j=2,\ldots,k$ with $g_1,\ldots,g_k,r_1,\ldots,r_k$ all analytic on $\Omega \times \pdom_{\pi}$, where $\pdom_{\pi} = \bigcap_{j=1}^k \pdom_k \subset \pdom$ is a neighborhood of $\pi$. Substitute $g_1$ in \eqref{eq:fdivision} with the decomposition of $g_1$ in \eqref{eq:gdivision} to get 
    \begin{equation*}
        f(z,p) = \big(g_2(z,p) u_2(z,p) + r_2(z,p)\big) u_1(z,p) + r_1(z,p).
    \end{equation*}
    Recursively replacing $g_{j-1}$ in the above equation with $g_{j-1}(z,p) = g_j(z,p) u_j(z,p) + r_j(z,p)$ for $j=3,\ldots,k+1$ gives
    \begin{equation*}
    \begin{aligned}
            f(z,p)=\underbrace{g_k(z,p)}_{\mbox{$=:g(z,p)$}} &\underbrace{u_k(z,p)\cdots u_1(z,p)}_{\mbox{$=:u(z,p)$}} \\
            &{} +\underbrace{r_k(z,p)u_{k-1}(z,p)\cdots u_1(z,p)+ \cdots +r_2(z,p)u_1(z,p)+r_1(z,p)}_{\mbox{$=:r(z,p)$}}.
    \end{aligned}
    \end{equation*}
    This expression amounts to
    \begin{equation*}
        f(z,p) = g(z,p) u(z,p) + r(z,p),
    \end{equation*}
    where $g,r$ are analytic in $\Omega \times \pdom_{\pi}$ and $r(\cdot,p)$ is a polynomial with degree less than $m$ for all $p \in \pdom_{\pi}$. Hence, we have shown so far that a sought-after decomposition $f = g_{\pi} u + r_{\pi}$ exists locally around $\pi$. We complete the proof by an analytic continuation argument. Consider $\pi, \pi' \in \pdom$ such that $\pdom_{\pi} \cap \pdom_{\pi'} = \pdom' \neq \emptyset$. On $\Omega \times \pdom'$ the two representations $f = g_{\pi} u + r_{\pi}$ and $f = g_{\pi'} u + r_{\pi'}$ coincide. Hence, we have
    \begin{equation*}
        (g_{\pi} - g_{\pi'}) u = r_{\pi'} - r_{\pi}
    \end{equation*}
    on $\Omega \times \pdom'$. In particular, we have for all $p \in \pdom'$ that
    \begin{equation*}
         g_{\pi}(\cdot,p) - g_{\pi'}(\cdot,p) = \frac{r_{\pi'}(\cdot,p) - r_{\pi}(\cdot,p)}{u(\cdot,p)}
    \end{equation*}
    is analytic in $\Omega$.\ \ Thus either the zeros of $r_{\pi'}(\cdot,p) - r_{\pi}(\cdot,p)$ and $u(\cdot,p)$ must cancel (which is not possible because the degree of $u$ is higher and all of its zeros are in $\Omega$) or both sides of the equation are zero, which is the case here. This means that $g_{\pi}(\cdot,p) = g_{\pi'}(\cdot,p)$ and $r_{\pi}(\cdot,p) = r_{\pi'}(\cdot,p)$ for all $p\in \pdom'$.\ \ By analytic continuation we then get that $g_{\pi}=g_{\pi'}$ and $r_{\pi}=r_{\pi'}$ on $\pdom_{\pi} \cup \pdom_{\pi'}$. This agreement also shows uniqueness of the derived decomposition, as analytic continuations are unique. For an arbitrary $\pi'' \in \pdom$ we can choose a path that connects $\pi$ and $\pi''$ (since $\pdom$ is a domain, and therefore connected) and repeat a similar argument with intersecting neighborhoods along that path to show that $g_{\pi} = g_{\pi''}$ and $r_{\pi} = r_{\pi''}$. Since we can do this for any $\pi'' \in \pdom$, we ultimately get that $f = g_{\pi} u + r_{\pi}$ in the entire set $\Omega \times \pdom$.
\end{proof}
\subsection{Proof of Theorem~\ref{theorem:parametricKeldysh}}
\label{sec:parametricKeldyshProof}
\begin{proof}
     First, we note that we can write 
    \begin{equation} \label{eq:Tinv}
        \bT(z,p)^{-1} = \frac{\operatorname{adj}\left( \bT(z,p) \right)}{\det\left( \bT(z,p) \right)},
    \end{equation}
    where the adjugate $\operatorname{adj}\left( \bT(z,p) \right)$ is a matrix-valued analytic function in $\Omega \times \pdom$. Similarly, $\det(\bT(z,p))$ is a scalar-valued analytic function in $\Omega \times \pdom$ having exactly $m$ zeros (counting order) in $\Omega$ for all $p \in \pdom$ due to the assumption that $\bT$ has exactly $m$ eigenvalues (counting algebraic multiplicity) in $\Omega$ for all $p \in \pdom$. To it we can thus apply Lemma~\ref{lemma:generalizedWeierstrass}, giving the decomposition
    \begin{equation}
        \label{eq:detdecomposition}
        \det\left( \bT(z,p) \right) = u(z,p) @h(z,p),
    \end{equation}
    where $u(\cdot,p)$ is a polynomial, $h$ is analytic, and $h(z,p) \neq 0$ in $\Omega \times \pdom$. Thus the function
    \begin{equation*}
        \bP(z,p) = \frac{\operatorname{adj}\left( \bT(z,p) \right)}{h(z,p)}
    \end{equation*}
    is a matrix-valued function analytic in $\Omega \times \pdom$.\ \  Writing $u$ as the product of shifted Weierstrass polynomials, as introduced in \eqref{eq:uproduct}, we can  apply Lemma~\ref{lemma:f2division} to each entry of $\bP$ to obtain a decomposition of the form
    \begin{equation*}
        \bP(z,p) = \bN(z,p)@u(z,p) + \bR(z,p),
    \end{equation*}
    where $\bN$ and $\bR$ are unique functions analytic in $\Omega \times \pdom$ and $\bR(\cdot,p)$ is a matrix-valued polynomial for all $p \in \pdom$.\ \ Thus we can recast~\eqref{eq:Tinv} as
    \begin{equation*}
        \bT(z,p)^{-1} = \frac{\bP(z,p)}{u(z,p)} = \frac{\bN(z,p) u(z,p) + \bR(z,p)}{u(z,p)} = \bN(z,p) + \frac{\bR(z,p)}{u(z,p)}.
    \end{equation*}
    For each $p \in \pdom$, $u(\cdot,p)$ is a polynomial, as are the entries of $\bR(\cdot,p)$, and hence $\bR(\cdot,p)/u(\cdot,p)$ is a matrix-valued rational function and $\bN(\cdot,p)$ is analytic in $\Omega$.\ \ Lem\-ma~\ref{lemma:f2division} ensures that the degree of $\bR(\cdot,p)$ is strictly less than the degree of $u(\cdot,p)$, so
    \begin{equation} \label{eq:HRu}
        \bH(z,p) = \frac{\bR(z,p)}{u(z,p)}
    \end{equation}
   is a strictly proper rational function.  To connect this new form of $\bH(\cdot,p)$ 
    with the eigenvector matrices $\bV(p)$ and $\bW(p)$, and the Jordan matrix $\bJ(p)$, we seek to reconcile~\eqref{eq:HRu} with the pointwise Keldysh decomposition.\ \ For each fixed $p\in\pdom$, write the pointwise Keldysh decomposition (along the lines of Proposition~\ref{proposition:pointwiseKeldysh}, generalized based on Theorem~\ref{theorem:generalKeldysh}) as
    \begin{equation*}
        \bT(z,p)^{-1} = \tbH(z,p) + \tbN(z,p),
    \end{equation*}
    where
    \begin{equation} \label{eq:Htilde}
        \tbH(z,p) = \bV(p) ( z\bI - \bJ(p))^{-1} \bW(p)^*.
    \end{equation}
    For each fixed $p \in \pdom$, and let $z_1,\ldots,z_k \in \Omega$ denote the eigenvalues of $\bT(\cdot,p)$. For each $z_j$, choose an open disk $C_j$ centered at $z_j$ that is sufficiently small that $z_j$ is the only eigenvalue of $\bT(\cdot,p)$ in $C_j$, $\overline{C}_j\subset \Omega$, and the functions $\bN(\cdot,p)$ and $\tbN(\cdot,p)$ are analytic on $C_j$. Note that both $\bH(\cdot,p)$ and $\tbH(\cdot,p)$ are strictly proper rational functions. We will show that these functions coincide by comparing their Laurent series coefficients. First, for $j=1,\ldots,k$ and $i=0,\ldots,m-1$ consider the equivalent expressions
    \begin{align*}
        (z-z_j)^i@@\bT(z,p)^{-1} 
        &= (z-z_j)^i @@\bH(z,p) + (z-z_j)^i @@\bN(z,p) \\
        &= (z-z_j)^i @@\tbH(z,p) + (z-z_j)^i @@\tbN(z,p).
    \end{align*}
   Integrate the equivalent expressions on the right-hand side along $\partial C_j$, substituting~\eqref{eq:Htilde} in the latter case, to get
    \begin{equation}
        \label{eq:LaurentCoefficientIntegrals}
        \int_{\partial C_j} (z-z_j)^i@@ \bH(z,p)\, \dop z = \int_{\partial C_j} (z-z_j)^i @@\bV(p) ( z\bI - \bJ(p))^{-1} \bW(p)^*\, \dop z.
    \end{equation}
    In particular, the terms $(z-z_j)^i @@\tbN(z,p)$ and $(z-z_j)^i @@\bN(z,p)$ vanish due to Cauchy's theorem, since they are analytic. The integrals in \eqref{eq:LaurentCoefficientIntegrals} are exactly the Laurent series coefficients of the principal parts of $\bH(\cdot,p)$ and $\tbH(\cdot,p)$ centered at $z_j$, and hence the principal parts of the rational functions $\bH(\cdot,p)$ and $\tbH(\cdot,p)$ coincide. We note that the polynomial part of the Laurent series for $\bH(\cdot,p)$ and $\tbH(\cdot,p)$ must be zero, because they are strictly proper rational functions. Overall, we can therefore conclude that $\bH(\cdot,p)$ and $\tbH(\cdot,p)$ must coincide in $\Omega$, and therefore, in particular, we have that
    \begin{equation*}
        \bH(z,p) = \bV(p) ( z\bI - \bJ(p))^{-1} \bW(p)^*
    \end{equation*}
    for all $p \in \pdom$.\ \ The statement about analyticity of $\bV,\bW$ and $\bJ$ (when all eigenvalues are simple) follows from the implicit function theorem, as discussed in \cite{andrew1993}.
\end{proof}

\bibliographystyle{plainurl}
\bibliography{references}

\end{document}